\documentclass{amsart}
\usepackage{amsmath,amssymb,mathrsfs,euscript,color,stmaryrd,mathabx}
\usepackage[colorlinks]{hyperref}
\usepackage{pdfsync}
\usepackage[backgroundcolor=green!40, linecolor=green!40]{todonotes}
\input{xypic}

\newcounter{makeconstant}
\newenvironment{makeconstant}%
{\refstepcounter{makeconstant}}%
{}

\newcounter{nonumber}

\theoremstyle{plain}
\newtheorem{theorem}[equation]{Theorem}
\newtheorem{lemma}[equation]{Lemma}
\newtheorem{corollary}[equation]{Corollary}
\newtheorem{proposition}[equation]{Proposition}
\newtheorem{theoremn}[nonumber]{Theorem}

\theoremstyle{definition}
\newtheorem{example}[equation]{Example}

\newtheorem{remark}[equation]{Remark}

\newtheorem{notitle}[equation]{}

\numberwithin{equation}{section} 

\def\CC{\mathbf{C}}
\def\FF{\mathbf{F}}
\def\GG{\mathbf{G}}
\def\NN{\mathbf{N}}

\def\RR{\mathbf{R}}
\def\ZZ{\mathbf{Z}} 

\def\A{{\rm A}}

\def\C{{\rm C}}

\def\E{{\rm E}}
\def\F{{\rm F}}
\def\G{{\rm G}}
\def\H{{\rm H}}

\def\J{{\rm J}}

\def\L{{\rm L}}
\def\M{{\rm M}}
\def\N{{\rm N}}
\def\O{{\rm O}}
\def\P{{\rm P}}
\def\Q{{\rm Q}}
\def\R{{\rm R}}

\def\T{{\rm T}}
\def\U{{\rm U}}
\def\V{{\rm V}}

\def\Aa{\mathcal{A}}

\def\Ee{\mathscr{E}}
\def\Ff{\mathcal{F}}
\def\Gg{\mathcal{G}}
\def\Hh{\mathscr{H}}

\def\Ll{\mathcal{L}}

\def\Nn{\mathcal{N}}

\def\Pp{\mathcal{P}}

\def\Rr{\mathfrak{R}}

\def\Tt{\mathcal{T}}

\def\Vv{\mathcal{V}}
\def\Ww{\mathcal{W}}

\def\Zz{\mathscr{Z}}

\def\ee{\mathrm{e}}
\def\hh{\mathfrak{h}}

\def\vv{\mathrm{v}}

\def\so{\mathsf{o}}
\def\an{\mathsf{an}}

\def\d{\delta}
\def\e{\varepsilon}
\def\bee{{\boldsymbol\varepsilon}}
\def\g{\gamma}
\def\vphi{\varphi}
\def\ii{{\mathrm i}}

\def\l{\lambda}

\def\o{\mathfrak{o}}
\def\p{\mathfrak{p}}
\def\s{\sigma}

\def\w{\varpi}
\def\z{\zeta}

\def\La{\Lambda}

\def\om{\omega}

\def\({\left(}
\def\){\right)}
\def\>{\geqslant}
\def\<{\leqslant}
\def\le{\leqslant}
\def\ge{\geqslant}
\def\lvert{\left\vert}
\def\rvert{\right\vert}

\def\mid{:}

\def\Hom{\operatorname{Hom}}
\def\End{\operatorname{End}}
\def\Aut{\operatorname{Aut}}

\def\GL{\operatorname{GL}}
\def\Gal{\operatorname{Gal}}

\def\Im{\operatorname{Im}}

\def\Res{\operatorname{Res}}
\def\Ind{\operatorname{Ind}}

\def\cind{\operatorname{c-Ind}}

\def\Jord{\operatorname{Jord}}
\def\Mod{\operatorname{Mod-}}
\def\mrank{\operatorname{-rank}}
\def\Cent{\operatorname{Cent}}

\def\Red{\operatorname{Red}}
\def\IRed{\operatorname{IRed}}
\def\diag{\operatorname{diag}}

\def\st{\operatorname{st}}
\def\GSp{\operatorname{GSp}}
\def\Sp{\operatorname{Sp}}
\def\GO{\operatorname{GO}}
\def\SO{\operatorname{SO}}
\def\GSO{\operatorname{GSO}}
\def\SL{\operatorname{SL}}

\def\hG{{\widehat{\G}}}
\def\tGg{{\widetilde{\Gg}}}
\def\tLl{{\widetilde{\Ll}}}
\def\tPp{{\widetilde{\Pp}}}
\def\BJ{{\bf J}}
\def\ov#1{{\overline{#1}}}
\def\la{\langle}
\def\ra{\rangle}
\def\boI{{\boldsymbol{1}}}

\def\ie{i.e.\ }
\def\ignore#1{\relax}

\begin{document}

\title{On depth zero~$\L$-packets for classical groups}
\author{Jaime Lust}\address{(JL) Department of Mathematics, 
University of Iowa, Iowa City, IA 52242, USA}
\author{Shaun Stevens}\address{(SS) School of Mathematics,
University of East Anglia, Norwich NR4 7TJ, UK}
\date{\today}

\thanks{SS was supported by EPSRC grants EP/G001480/1 and EP/H00534X/1}


\begin{abstract}
By computing reducibility points of parabolically induced representations, 
we construct, to within at most two unramified quadratic characters, the 
Langlands parameter of an arbitrary depth zero irreducible cuspidal 
representation~$\pi$ of a classical group (which may be not-quasi-split) over 
a nonarchimedean local field of odd residual characteristic. From this, we can explicitly
describe all the irreducible cuspidal representations in the union of one, two, 
or four~$L$-packets, containing~$\pi$. These results generalize the work of 
DeBacker--Reeder (in the case of classical groups) from regular to arbitrary tame 
Langlands parameters.
\end{abstract}

\maketitle

\section{Introduction}

The representation theory of~$p$-adic groups has largely been
motivated, over the last half-century, by the Langlands
conjectures, seeking an understanding of the absolute Galois (or Weil) groups of
local and global fields. Many parts of the local conjectures are now
theorems, notably for representations of~$\GL_n$~\cite{HT,H00},
of~$\SL_n$~\cite{GK,HS} and, more recently, of classical groups~\cite{Arthur,Mok,KMSW}.

At the same time as having local Langlands correspondences, one would
like to be able to use them to translate fine arithmetical data
between representations of~$p$-adic groups and representations of the
local Weil group. To this end, one seeks to make the correspondence
explicit/effective. For~$\GL_n$, this has been the subject of a series
of papers by Bushnell--Henniart~\cite{BHet1,BHet2,BHet3,BHmem}; for
other groups, work has concentrated on \emph{regular depth zero
irreducible cuspidal representations}~\cite{DeBR,Kal1,Kal2} and \emph{epipelagic
irreducible cuspidal representations}~\cite{Reeder,GReeder,Kal3,RY,Kal4,BHepi}, 
with the most general work by Kaletha~\cite{Kal5} on \emph{regular cuspidal representations}.

In this spirit, we look here at depth zero irreducible cuspidal representations of
a classical group~$\G$ -- by which we mean a symplectic, (special) orthogonal,
or unitary group, which may be non-quasi-split -- over a
nonarchimedean locally compact local field \emph{of odd residual
characteristic} (this is the only restriction on the field). When
these representations are also \emph{regular} (more precisely, the
corresponding Langlands parameter is \emph{tame regular semisimple in
general position}), these have already been considered, for more
general groups but with some conditions on the field, by
DeBacker--Reeder~\cite{DeBR} and Kaletha~\cite{Kal1,Kal2}; however,
our approach here is different, and allows us to treat all depth zero
irreducible cuspidal representations. Thus our work, and methods, are 
complementary to those of~\cite{Kal5}.

Given a Langlands parameter for~$\G$, the Langlands correspondence
should determine an \emph{$\L$-packet} of irreducible smooth complex
representations of~$\G$. These representations should share many
properties; for example, they should have all the
same~$\L$-functions, at least where these have been defined. Since, by
the results of Shahidi~\cite{Sh90},
poles of~$\L$-functions correspond to reducibility points of parabolic
induction, we detect representations in the same~$\L$-packet by
computing these reducibility points, and this does not require, for
example, genericity of the representation.

For now, we are not able completely to compute reducibility points,
but only up to \emph{twist by a certain unramified character} (see below for
more details). However, an even-dimensional irreducible tame
representation of the Weil group is symplectic if and only if this
unramified twist is orthogonal; thus, using the Langlands
correspondence, for example for symplectic groups, we can see which of
the twists must occur, and the only ambiguity is in the
reducibility points for quadratic characters of~$\GL_1$. In any case,
we are able to recover the irreducible cuspidal representations in the union of
either one, two or four~$\L$-packets.

This paper can be regarded as a first step in a programme to treat all
discrete series 
representations of classical groups -- see~\cite{BHS} for the case of
arbitrary irreducible cuspidal representations of symplectic groups. Depth zero is the
base case, since general irreducible cuspidal representations are built from a
``wild part'' and a depth zero part (see~\cite{S5}). In the depth zero
case, we avoid the complication of wild ramification; on the other
hand, the geometric complications arise essentially from the depth
zero part so that the results and techniques here already resolve many
difficulties for the general case.

\medskip

Now let us state our results more carefully; although we have
interpreted them above via the Langlands correspondence, they are in
fact results on the automorphic side. Let~$\F/\F_\so$ be an
extension of degree at most two of nonarchimedean local fields of odd
residual characteristic, and let~$\G$ be (the group of rational points
of) a symplectic, special orthogonal or unitary group over~$\F_\so$,
the connected component of the group of isometries of
an~$\F/\F_\so$-hermitian space; this group may be non-quasi-split. We
also write~$\Ww_\F$ for the Weil group of~$\F$ and~$\hG$ for the
complex dual group of~$\G$, acting naturally on a vector space of
dimension~$N_\hG$.

In their classification of discrete series representations of
(quasi-split)~$p$-adic classical groups~\cite{MT,Mo1,Mo4}, M\oe
glin--Tadi\'c use the notion of a \emph{Jordan set} attached to an
irreducible discrete series representation of~$\G$. For an
irreducible cuspidal representation~$\pi$ of~$\G$, this can be described via
the \emph{reducibility set}~$\Red(\pi)$ as follows. 

We denote by~$\Aa^\s(\F)$ the set of (equivalence
classes of) self-dual irreducible cuspidal representations of
some~$\GL_n(\F)$ (see Section~\ref{S.cuspidals}).
For~$\rho\in\Aa^\s(\F)$ there is at most one real 
number~$s=s_\pi(\rho)\ge 0$ such that the normalized parabolically
induced representation~$\Ind \rho|\det(\cdot)|_\F^s\otimes\pi$ is reducible,
where~$\lvert\,\cdot\,\rvert_\F$ is the normalized absolute value
on~$\F$; when there is no such real number (which can happen only for
even-dimensional special orthogonal groups and~$\rho$ a quadratic
character of~$\GL_1(\F)$), we set~$s_\pi(\rho)=0$. Then
\[
\Red(\pi)\ =\ 
\left\{(\rho,m)\mid \rho\in\Aa^\s(\F),
~m\in\NN\text{ with }2s_\pi(\rho)=m+1\right\}.
\]
M\oe glin proves in~\cite{Mo2} that, again for~$\pi$ irreducible cuspidal, the
Jordan set is
\begin{eqnarray}
\Jord(\pi)& =& 
\left\{(\rho,m)\mid \rho\in\Aa^\s(\F),
~m\in\NN\text{ with }2s_\pi(\rho)-(m+1)\in2\ZZ_{\ge 0}\right\} \notag\\
&=& \left\{(\rho,m)\mid (\rho,m_\rho)\in\Red(\pi)\text{ and
  }m_\rho-m\in2\ZZ_{\ge 0}\right\}, \label{eqn:noholes}
\end{eqnarray}
so that~$\Red(\pi)$ is the set of maximal elements of~$\Jord(\pi)$.

It is expected (and in at least when~$G$ is quasi-split, known -- see, for example,~\cite{Mo5}) that the Jordan set should
precisely predict the Langlands
parameter~$\vphi:\Ww_\F\times\SL_2(\CC)\to\hG\rtimes\Ww_\F$ whose~$\L$-packet~$\Pi_\vphi$
contains~$\pi$, by
\[
\vphi\ =\ \bigoplus_{(\rho,m)\in\Jord(\pi)} \vphi_\rho\otimes\st_m,
\]
where~$\vphi_\rho$ is the (irreducible) representation of the Weil
group~$\Ww_\F$ corresponding to~$\rho$ via the Langlands
correspondence for general linear groups, and~$\st_m$ is
the~$m$-dimensional irreducible representation of~$\SL_2(\CC)$. In
particular, writing~$n_\rho$ for the unique natural number such
that~$\rho$ is a representation of~$\GL_{n_\rho}(\F)$, we should have
\begin{equation}\label{eqn:expected0}
\sum_{(\rho,m)\in\Jord(\pi)}mn_\rho\ =\ N_{\hG}.
\end{equation}
By~\eqref{eqn:noholes}, this equality is equivalent to
\begin{equation}\label{eqn:expected}
\sum_{\rho\in\Aa^\s(\F)}\left\lfloor\(s_\pi(\rho)\)^2\right\rfloor n_\rho\ =\ N_{\hG},
\end{equation}
where~$\lfloor x\rfloor$ denotes the greatest integer not exceeding~$x$.
Note that almost all terms in this sum are zero since~$s_\rho(\pi)<1$
for all but finitely many~$\rho\in\Aa^\s(\F)$.

\medskip

Suppose now that the representation~$\pi$ is of depth zero; equivalently, 
the Langlands parameter is \emph{tame} (\ie~trivial on restriction to the
wild inertia subgroup of~$\Ww_\F$). 
For clarity of exposition, we specialize temporarily to the case of a 
symplectic group~$\G$, in which case~$\hG$ is a special orthogonal group 
with~$N_{\hG}$ odd. On the other hand, by a result of Blondel~\cite{Bl}, there are 
self-dual irreducible cuspidal representations of~$\GL_n(\F)$ only for~$n$ 
even or~$n=1$; in the latter case, we get the pair of unramified 
characters of order dividing two, and the pair of (tamely) ramified
quadratic characters. Since~$N_{\hG}$ is odd, equation~\eqref{eqn:expected0} implies 
that there is exactly one pair for which the multiplicities of the two characters 
in~$\vphi|_{\Ww_\F}$ have the same parity; we denote by~$\vphi'$ the Langlands 
parameter obtained from~$\vphi$ by exchanging the multiplicities of the two characters
in this pair. (In particular, we have~$\vphi'=\vphi$ when the multiplicities 
are equal.)

\begin{theoremn}
\begin{enumerate}
\item Given a tame Langlands parameter~$\vphi$ for a symplectic group as above, there is an explicit 
description of the cuspidal representations in the union 
of~$\L$-packets~$\Pi_\vphi\cup\Pi_{\vphi'}$.
\item Conversely, given a depth-zero cuspidal irreducible representation~$\pi$ 
of a symplectic group as above, there is an explicit description of the pair~$\{\vphi,\vphi'\}$ of 
tame Langlands parameters, such that~$\pi\in\Pi_\vphi\cup\Pi_{\vphi'}$.
\end{enumerate}
\end{theoremn}

We remark that, in the situation of \emph{regular} depth zero
irreducible cuspidal representations, the multiplicities of the characters 
are all at most one, so that~$\vphi'=\vphi$; thus we recover the description 
of the representations in an~$L$-packet consisting solely of cuspidal representations 
from~\cite{DeBR} in this case.

\medskip

We return to the case of depth zero representations of a general classical group~$\G$ 
and describe the result here, which is a reinterpretation of the theorem above in terms 
of the set~$\Red(\pi)$. More precisely,
denote by~$[\rho]$ the \emph{inertial equivalence class}
of~$\rho\in\Aa^\s(\F)$, that is, the set of unramified twists of~$\rho$; note
that~$[\rho]\cap\Aa^\s(\F)=\{\rho,\rho'\}$ consists of exactly two
(inequivalent) representations. We have~$\rho'=\rho\chi$, for~$\chi$ an unramified 
character with~$\rho\chi^2=\rho$; since~$\rho$ has depth zero, the 
character~$\chi$ has order~$2n_\rho$. 

Here, we compute the \emph{inertial reducibility multiset}
\[
\IRed(\pi)\ =\ \{\!\{([\rho],m)\mid(\rho,m)\in\Red(\pi)\}\!\}.
\]
This is often in fact a set: since~$\pi$ has depth zero, the only inertial 
classes~$[\rho]$ which can occur with multiplicity are the quadratic characters 
of~$\GL_1(\F)$. Indeed, for~$\rho\in\Aa^\s(\F)$ of depth zero with~$n_\rho>1$  
and~$\rho'$ its self-dual unramified twist, the exterior square~$L$-function of exactly one 
of~$\rho,\rho'$ has a pole at~$s=0$, while for the other representation it is the 
symmetric square~$L$-function which has a pole (the same comments apply to the Asai 
and twisted Asai~$L$-functions when~$\F/\F_\so$ is quadratic); thus the parity of~$m$, 
such that~$(\rho,m)\in\Red(\pi)$ should be independent of~$\pi$ (\ie depend only on~$\rho$), 
and the parity will be the opposite of that for~$\rho'$. Moreover, this means that by 
computing~$\IRed(\pi)$, we in fact know all elements of~$\Red(\pi)$ apart from those 
associated to characters of~$\GL_1(\F)$ of order at most two, where an ambiguity may remain.

The results we prove here can be described by the following:

\begin{theoremn} Let~$\pi$ be a depth zero irreducible cuspidal representation
  of~$\G$.
\begin{enumerate}
\item Equation~\eqref{eqn:expected} holds.
\item The multiset~$\IRed(\pi)$ can be computed explicitly from the
  local data defining~$\pi$ as a compactly induced representation. 
\item The set of irreducible cuspidal representations~$\pi'$ of~$\G$
  with~$\IRed(\pi)=\IRed(\pi')$ can be described explicitly in terms of
  the local data defining~$\pi$. Moreover, the number of such representations
  is the expected number in one, two or four~$L$-packets, this number depending
  again on the local data.
\end{enumerate}
\end{theoremn}
For a discussion of the expected number of irreducible cuspidal representations in an~$\L$-packet,
see the beginning of Section~\ref{S.examples}. One also needs to take care with this in 
the case of even orthogonal groups (see Example~\ref{ex:ortho}).

The computation of reducibility points required for this theorem is
achieved using Bushnell--Kutzko's theory of covers~\cite{BK1}, together with
results of Blondel~\cite{Blb}, which translate the problem to the
computation of parameters in the Hecke algebra of a cover (see
Sections~\ref{S.covers}--\ref{S.reduction}). These covers were
constructed in~\cite{MiS}, and the parameters can be computed
using Morris's explicit description of depth zero irreducible cuspidal
representations (see Section~\ref{S.cuspidals}), together with results
of Lusztig on representations of finite reductive groups (see
Section~\ref{S.computation}). However, care must be taken since the
finite reductive groups which occur do not, in general, have connected
centre. There is particular difficulty for (even-dimensional) special
orthogonal groups and the results we obtain here may be of independent
interest; in particular, we compute when an irreducible cuspidal
representation of an even-dimensional special orthogonal group over a
finite field extends to the full orthogonal group (see
Proposition~\ref{prop:orthogonalextension}), generalizing results of
Lusztig and Waldspurger. All these ingredients are then put together
to prove the theorem in Sections~\ref{S.synthesis}
and~\ref{S.examples}; the latter includes some illustrative examples.

\subsection*{Acknowledgements}
The research of SS was supported by EPSRC grants EP/G001480/1 and EP/H00534X/1. He
would like to thank Meinolf Geck, and particularly Marc Cabanes
for his patience in explaining Deligne--Lusztig theory -- any
remaining mistakes are entirely the authors'. He would also like to
thank Corinne Blondel and Guy Henniart for their patience in waiting
for this paper to get written up.

\section{Notation and background}
\label{S.notation}

We fix some notation for the rest of the paper (with the exception of
Section~\ref{S.computation}, whose notation is
independent). Let~$\F_\so$ be a locally compact nonarchimedean local
field of \emph{odd} residual characteristic~$p$, and let~$\F/\F_\so$
be an extension of degree at most~$2$. We write~$\s:\l\mapsto\ov\l$
for the generator of the Galois group of~$\F/\F_\so$. For~$\E$ any
field containing~$\F_\so$, we write~$\o_\E$ for its ring of
integers,~$\p_\E$ for its maximal ideal, and~$k_\E=\o_\E/\p_\E$ for
its residue field, of cardinality~$q_\E$; in particular we
abbreviate~$q=q_{\F}$. We also abbreviate~$\o_\so=\o_{\F_{\so}}$
etc. We also fix a uniformizer~$\w_\F$ of~$\F$ such
that~$\ov{\w_\F}=-\w_\F$ if~$F/F_\so$ is quadratic ramified,
and~$\ov{\w_\F}=\w_\F$ otherwise, and
write~$\N_{\F/\F_\so}:\F^\times\to\F_\so^\times$ for the norm map,
which is given by~$\l\mapsto\l\ov\l$ if~$[\F:\F_\so]=2$.

We fix a sign~$\e=\pm 1$, and let~$(\V,h)$ be a
nondegenerate~$\F/\F_\so$-$\e$-hermitian space of Witt index~$N$ and
dimension~$2N+N^\an$; thus~$\V$ is an~$\F$-vector space, the form~$h$
satisfies
\[
h(\l v,\mu w)= \l\ov\mu h(v,w) = \e\l\ov\mu \overline{h(w,v)},
\qquad\hbox{for }v,w\in\V,\ \l,\mu\in\F,
\]
and we have a Witt decomposition
\[
\V=\V^-\oplus\V^\an\oplus\V^+,
\] 
with~$\dim_\F\V^\pm=N$ and~$\dim_\F\V^\an=N^\an$, such that the
restriction of~$h$ to~$\V^\pm$ is totally isotropic, while its
restriction~$h^\an$ to~$\V^\an$ is anisotropic. We denote
by~$\H=\H^-\oplus\H^+$ the hyperbolic plane; that is,~$\H^\pm$ is
a~$1$-dimensional~$\F$-vector space with basis~$\ee_\pm$ and~$\H$ is
equipped with the form~$h_\H$ given by
\[
h_\H(\l_-\ee_- + \l_+\ee_+ , \mu_-\ee_- + \mu_+\ee_+ ) = \l_-\ov{\mu_+} + \e\l_+\ov{\mu_-},
\qquad\hbox{for }\l_\pm,\mu_\pm\in\F.
\]
Thus the restriction of~$h$ to~$\V^-\oplus\V^+$ is (isometric to) an
orthogonal direct sum of~$N$ copies of~$\H$. We choose a Witt basis
for~$\V$, that is:~$\ee_1^+,\ldots,\ee_N^+$ a basis for~$\V^+$, with
dual basis~$\ee_1^-,\ldots,\ee_N^-$ for~$\V^-$,
and~$\ee_1^\an,\ldots,\ee_{N^\an}^\an$ a basis~$\V^\an$ with respect
to which~$h^\an$ has diagonal Gram matrix. We order this basis
\[
\ee_N^-,\ldots,\ee_1^-,\ee_1^\an,\ldots,\ee_{N^\an}^\an,\ee_1^+,\ldots,\ee_N^+.
\]

For~$n\ge 0$, we denote by~$n\H$ the orthogonal direct sum of~$n$
copies of~$\H$, and put
\[
\V_n=\V\oplus n\H
\] 
with the form~$h_n=h\oplus h_\H\oplus\cdots\oplus h_\H$, so that the
decomposition above is orthogonal and we have a Witt decomposition
\[
\V_n=\V_n^-\oplus\V^\an\oplus\V_n^+,\qquad\hbox{with }\V_n^\pm=\V^\pm\oplus n\H^\pm.
\]
Thus~$\(\V_n\mid n\ge 0\)$ is a Witt tower
over~$\V_0=\V$. Writing~$\ee^\pm_{N+i}$ for the image in~$\V_n$
of~$\ee_\pm$ in the~$i^{\rm th}$ copy of~$\H$, the space~$\V_n$ has
the ordered Witt basis
\[
\ee_{N+n}^-,\ldots,\ee_1^-,\ee_1^\an,\ldots,\ee_{N^\an}^\an,\ee_1^+,\ldots,\ee_{N+n}^+.
\]

For~$n\ge 0$, we put~$\G_n^+=\U(\V_n)$, the group of~$\F_\so$-rational
points of the reductive algebraic group over~$\F_\so$ determined
by~$(\V_n,h_n)$, so that
\[
\G_n^+=\{g\in\Aut_\F(\V_n)\mid h_n(gv,gw)=h_n(v,w)\hbox{ for all }v,w,\in\V\};
\]
thus~$\G_n^+$ is (the group of points of) a unitary, symplectic or
(full) orthogonal group. We also put~$\G_n=\U(\V_n)^\so$, the group
of~$\F_\so$-rational points of the connected component, so that
\[
\G_n=\{g\in\G_n^+\mid \N_{\F/\F_\so}\textstyle\det_\F(g)=1\};
\]
thus~$\G_n=\G_n^+$ unless~$\G_n^+$ is an orthogonal group
(so~$\F=\F_\so$ and~$\e=1$), in which case~$\G_n$ is the special
orthogonal group, of index~$2$ in~$\G_n^+$. 
We will abbreviate~$\G=\G_0$ and~$\G^+=\G_0^+$. 

The stabilizer in~$\G_n$ of the decomposition
\[
\V_n = n\H^-\oplus \V\oplus n\H^+
\]
is a Levi subgroup~$\M_n$ of~$\G_n$, which is standard with respect to 
the chosen Witt basis, and we have an isomorphism~$\M_n\simeq \GL_n(\F)\times \G$ 
given by~$g\mapsto (g|_{n\H^-},g|_{\V})$; moreover, the
stabilizer of the subspace~$n\H^-$ is a standard parabolic
subgroup~$\P_n$ of~$\G_n$, with Levi component~$\M_n$. Thus, writing
elements of~$\G_n$ as matrices with respect to the Witt basis, the
group~$\P_n$ is block upper triangular and~$\M_n$ is block diagonal.

\medskip

We end this section with a description of the maximal parahoric
subgroups of~$\G$ and of their reductive quotients (see also~\cite{M0}).
For~$\L$
an~$\o_\F$-lattice in~$\V$, we denote by~$\L^\#$ the dual lattice
\[
\L^\#=\{v\in\L\mid h(v,\L)\subseteq\p_\F\}.
\]
We say that~$\L$ is \emph{almost self-dual} if
\[
\L\supseteq \L^\# \supseteq\p_\F\L;
\]
in that case, the stabilizer~$\J=\J_\L$ in~$\G$ of~$\L$ is a maximal
compact subgroup of~$\G$, and every maximal compact subgroup arises in
this way for a unique self-dual lattice~$\L$. We write~$\J^1$ for the
pro-unipotent radical of~$\J$, that is the subgroup consisting of
those elements~$g$ which induce the identity map on the~$k_\F$-vector
spaces~$\ov\V_{(1)}:=\L/\L^\#$ and~$\ov\V_{(2)}:=\L^\#/\p_\F\L$.

The form~$h$ induces nondegenerate~$k_\F/k_\so$-forms on~$\ov\V_{(1)}$
and~$\ov\V_{(2)}$ by
\[
\begin{cases}
h_{\ov\V}(v+\L^\#,w+\L^\#) := h(v,w)+\p_\F, &\hbox{for }v,w\in\L,\\[5pt]
h_{\ov\V^\#}(v'+\p_\F\L,w'+\p_\F\L) := \w_\F^{-1}h(v',w')+\p_F, &\hbox{for }v',w'\in\L^\#.
\end{cases}
\]
The form~$h_{\ov\V_{(1)}}$ is~$\e$-hermitian, while the form~$h_{\ov\V_{(2)}}$
is~$(-\e)$-hermitian if~$\F/\F_\so$ is quadratic ramified,
and~$\e$-hermitian otherwise, by our choice of uniformizer. Thus we
get an induced map
\[
\J \to \U(\ov\V_{(1)}) \times \U(\ov\V_{(2)}),
\]
with kernel~$\J^1$, and hence the quotient~$\Gg=\Gg_\L=\J/\J^1$ is
naturally a subgroup of the finite reductive group~$\U(\ov\V_{(1)}) \times
\U(\ov\V_{(2)})$. In fact,~$\Gg$ identifies with the subgroup
\[
\left\{(g_1,g_2)\in \U(\ov\V_{(1)}) \times \U(\ov\V_{(2)}) \mid \N_{k_\F/k_\so}(\textstyle\det_{k_\F}(g_1)\textstyle\det_{k_\F}(g_2))=1\right\},
\]
which has connected component~$\Gg^\so=\U(\ov\V_{(1)})^\so \times
\U(\ov\V_{(2)})^\so$. We denote by~$\J^\so=\J^\so_\L$ the inverse image
in~$\J$ of~$\Gg^\so$; this is a parahoric subgroup of~$\G$ and~$\J$ is
its normalizer in~$\G$. It is \emph{not} always a maximal parahoric
subgroup of~$G$ (it is so if and only if neither factor~$\U(\ov\V_{(i)})^\so$ 
is a two-dimensional special orthogonal group)
but every maximal parahoric subgroup does arise in
this way. If either~$\F/\F_\so$ is quadratic ramified and the
orthogonal space among~$\ov\V_{(1)},\ov\V_{(2)}$ is non-zero,
or~$\F=\F_\so$,~$\e=1$ and both~$\ov\V_{(1)},\ov\V_{(2)}$ are non-zero,
then~$\J^\so$ has index~$2$ in~$\J$; otherwise we have~$\J=\J^\so$.

\medskip

Restricting first to the case of the hermitian space~$(\V^\an,h^\an)$,
there is a unique almost self-dual lattice~$\L_\an$ in~$\V^\an$, and
the corresponding group~$\G_\an=\U(\V^\an)^\so$ is compact and
normalizes the unique (maximal) parahoric
subgroup~$\J_\an^\so=\J_{\L_\an}^\so$, with connected
component~$\Gg_\an^\so$. We set
\[
\ov\V^\an_{(1)}=\L_\an/\L_\an^\#,\qquad \ov\V^\an_{(2)}=\L_\an^\#/\p_\F\L_\an,
\]
and~$N^\an_i=\dim_{k_\F}\ov\V^\an_{(i)}$, so
that~$N^\an=N^\an_1+N^\an_2$. Then we have the following possibilities
for~$\Gg_\an^\so=\U(\ov\V^\an_{(1)})^\so \times \U(\ov\V^\an_{(2)})^\so$,
where we write~$\SO(M_1,M_2,k_\F)$ for the special orthogonal group
with form of Witt index~$M_2$ and anisotropic part of
dimension~$M_1-M_2\le 2$:
\begin{itemize}
\item If~$\F=\F_\so$ and~$\e=-1$ then~$N^\an=0$ so~$\Gg_\an^\so$ is trivial.
\item If~$\F=\F_\so$ and~$\e=1$ then~$\Gg_\an^\so\simeq\SO(N^\an_1,0,k_\F)\times\SO(N^\an_2,0,k_\F)$, with~$N^\an_i\le 2$.
\item If~$\F/\F_\so$ is unramified quadratic, then~$\Gg_\an^\so \simeq\U(N^\an_1,k_\F/k_{\F_\so})\times\U(N^\an_2,k_\F/k_{\F_\so})$, the product of two unitary groups with~$N^\an_i\le 1$.
\item If~$\F/\F_\so$ is ramified quadratic then
\[
\Gg_\an^\so\simeq\begin{cases}
\SO(N^\an_1,0,k_\F)&\text{ if }\e=+1,\\
\SO(N^\an_2,0,k_\F)&\text{ if }\e=-1,
\end{cases}
\]
with~$N_i^\an\le 2$ and only one~$N_i^\an$ non-zero.
\end{itemize}

\medskip

Returning to the general case of the space~$(\V,h)$, the
\emph{standard} almost self-dual lattices are those of the following
form: for~$0\le N_1,N_2$ with~$N_1+N_2=N$, set
\[
\L_{N_1,N_2}:=\o_\F\ee_N^-\oplus\cdots\oplus\o_\F\ee_1^-\oplus\L^\an\oplus\o_\F\ee_1^+\oplus\cdots\oplus\o_\F\ee_{N_1}^+\oplus\p_\F\ee_{N_1+1}^+\oplus\cdots\oplus\p_\F\ee_N^+,
\]
where~$\L^\an$ is the unique almost self-dual lattice in~$\V^\an$. We
write~$\J_{N_1,N_2}$ 
for the stabilizer of~$\L_{N_1,N_2}$ and~$\J_{N_1,N_2}^\so$ for the corresponding
parahoric subgroup. Every almost self-dual lattice has the form~$g\L_{N_1,N_2}$, for
some~$g\in\G$ and a unique standard lattice~$\L_{N_1,N_2}$; thus every maximal
compact (respectively, maximal parahoric) subgroup is conjugate to a
unique standard one~$\J_{N_1,N_2}$ (respectively,~$\J_{N_1,N_2}^\so$). The choice of
Witt basis and the forms on~$\ov\V_{(1)},\ov\V_{(2)}$ then give us the
following identifications for the connected reductive
quotients~$\Gg_{N_1,N_2}^\so=\J_{N_1,N_2}^\so/\J_{N_1,N_2}^1$:
\begin{itemize}
\item If~$\F=\F_\so$ and~$\e=-1$ then
\[
\Gg^\so_{N_1,N_2}\ \simeq\ \Sp(2N_1,k_\F)\times\Sp(2N_2,k_\F).
\]
\item If~$\F=\F_\so$ and~$\e=1$ then
\[
\Gg^\so_{N_1,N_2}\ \simeq\ \SO(N_1+N^\an_1,N_1,k_\F)\times\SO(N_2+N^\an_2,N_2,k_\F).
\]
\item If~$\F/\F_\so$ is quadratic unramified then 
\[
\Gg^\so_{N_1,N_2}\ \simeq\ \U(2N_1+N^\an_1,k_\F/k_\so)\times\U(2N_2+N^\an_2,k_\F/k_\so).
\]
\item If~$\F/\F_\so$ is quadratic ramified then
\[
\Gg^\so_{N_1,N_2}\ \simeq\ 
\begin{cases}
\SO(N_1+N^\an_1,N_1,k_\F)\times\Sp(2N_2,k_\F)&\text{ if }\e=+1, \\
\Sp(2N_1,k_\F)\times\SO(N_2+N^\an_2,N_2,k_\F)&\text{ if }\e=-1.
\end{cases}
\] 
\end{itemize}
Writing~$\ov\H$ for hyperbolic space over~$k_\F$, we can unify these by writing
\[
\Gg^\so_{N_1,N_2}\ \simeq\ \Gg^{(1)}_{N_1}\times\Gg^{(2)}_{N_2},
\]
with~$\Gg^{(i)}_{N_i}=\U(N_i\ov\H\oplus\ov\V^\an_{(i)})^\so$,
for~$i=1,2$.

We note that~$\J_{N_1,N_2}^\so$ is a maximal parahoric subgroup except where one of the factors here is~$\SO(1,1,k_\F)$ but~$\G$ is not itself a~$2$-dimensional special orthogonal group; that is, in the following cases:
\begin{itemize}
\item $\F=\F_\so$,~$\e=1$,~$(N,N^\an)\ne (1,0)$,
 with~$(N_i,N^\an_i)=(1,0)$, for~$i=1$ or~$2$;
\item $\F/\F_\so$ is quadratic ramified,~$N^\an=0$
 and~$(\e,N_1)=(1,1)$ or~$(\e,N_2)=(-1,1)$.
\end{itemize}

\section{Depth zero cuspidal representations}
\label{S.cuspidals}

In this section, we recall the classification of the depth zero
irreducible cuspidal representations of~$\GL_n(\F)$ and of the classical
group~$\G$, beginning with the former.

We write~$\Aa_n(\F)$ for the set of equivalence classes of irreducible
cuspidal representations of~$\GL_n(\F)$ and put~$\Aa(\F)=\bigcup_{n\ge
 1}\Aa_n(\F)$. We will abuse notation by writing~$\rho\in\Aa(\F)$ to
mean~$\rho$ is an irreducible cuspidal representation of
some~$\GL_n(\F)$, where~$n=n_\rho$ is of course uniquely determined
by~$\rho$. For~$\rho\in\Aa_n(\F)$, we denote by~$\rho^\s$ the
representation 
\[
\rho^\s(g) = \rho(\s(g^{-1})^\T), \qquad\hbox{ for }g\in\GL_n(\F),
\]
where~$\s(g)$ denotes the matrix obtained by applying the
generator~$\s$ of~$\Gal(\F/\F_\so)$ to each entry, and~$g^\T$ denotes
the transpose matrix. We say that~$\rho$ is \emph{self-dual}
if~$\rho^\s\simeq\rho$, and write~$\Aa_n^\s(\F)$ for the set of
equivalence classes of self-dual irreducible cuspidal representations
of~$\GL_n(\F)$, and~$\Aa^\s(\F)=\bigcup_{n\ge 1}\Aa_n^\s(\F)$ for the
set of equivalence classes of self-dual representations in~$\Aa(\F)$.

We do not recall here the general notion of depth, only that a
representation~$\rho\in\Aa(\F)$ is said to be of \emph{depth zero}
if it has fixed vectors under the pro-unipotent radical of the maximal
parahoric subgroup~$\GL_{n_\rho}(\o_\F)$
of~$\GL_{n_\rho}(\F)$. We denote by~$\Aa_{[0]}(\F)$ the set of
equivalence classes of depth zero representations in~$\Aa(\F)$, and
by~$\Aa_{[0]}^\s(\F)$ the set of equivalence classes of self-dual
depth zero representations in~$\Aa(\F)$.

Any depth zero representation~$\rho\in\Aa_n(\F)$ can be written
\[
\rho\ =\ \cind_{\BJ_\rho}^{\GL_n(\F)} \La_\rho,
\]
where~$\BJ_\rho=\F^\times\GL_n(\o_\F)$ is the normalizer of the
maximal parahoric subgroup~$\J_\rho=\GL_n(\o_\F)$ of~$\GL_n(\F)$,
and~$\La_\rho$ is an irreducible representation of~$\BJ_\rho$ whose
restriction~$\l_\rho=\La_\rho|_{\J_\rho}$ is the inflation of an
irreducible cuspidal representation~$\tau_\rho$ of the reductive
quotient~$\J_\rho/\J_\rho^1\simeq\GL_n(k_\F)$. Moreover, the
(equivalence class of the) representation~$\tau=\tau_\rho$ is uniquely
determined by~$\rho$. Further,~$\rho$ is self-dual if and only
if~$\tau$ is self-dual; that is, denoting again by~$\s$ the generator
of~$\Gal(k_\F/k_\so)$ and by~$\tau^\s$ the
representation~$\tau^\s(g)=\tau_\rho(\s(g^{-1})^\T)$,
for~$g\in\GL_n(k_\F)$, we have~$\tau^\s\simeq\tau$.

The (equivalence classes of) irreducible cuspidal
representations~$\tau$ of~$\GL_n(k_\F)$ were first classified by
Green~\cite{Green}, and are parametrized by regular characters of the
multiplicative group of the degree~$n$ extension of~$k_\F$, 
or, equivalently (after making choices), by monic irreducible
degree~$n$ polynomials~$P=P_\tau\in k_\F[X]$ with~$P(0)\ne
0$. Writing~$\sigma(P)$ for the polynomial obtained by
applying~$\sigma$ to the coefficients of~$P$, the
representation~$\tau^\s$ then corresponds to the
polynomial~$P^\sigma(X):=\sigma(P(0))^{-1}X^{\deg(P)}\sigma(P)(1/X)$. Thus~$\tau$
is self-dual if and only if~$P_\tau=P_\tau^\sigma$. If~$k_\F=k_\so$,
such polynomials exist if and only if~$n=1$ or~$n$ is even
(see~\cite{Adler}); if~$k_\F/k_\so$ is quadratic, then such
polynomials exist if and only if~$n$ is odd
(see~\cite[\S5.4]{Kariyama}).

\medskip

Similarly, we write~$\Aa(\G)$ for the set of equivalence classes
of irreducible cuspidal representations of~$\G$, and~$\Aa_{[0]}(\G)$
for the subset of equivalence classes of depth zero representations. 
Then, for~$\pi\in\Aa_{[0]}(\G)$, we can write
\[
\pi\ =\ \cind_{\J_\pi}^{\G}\l_\pi,
\]
where~$\J_\pi=\J_{N_1,N_2}$ is the (compact open) normalizer of a
standard maximal parahoric subgroup~$\J^\so_{\pi}$ and~$\l_\pi$ is an
irreducible representation of~$\J_{\pi}$ whose
restriction~$\l^\so_\pi=\l_\pi|_{\J^\so_{\pi}}$ is a sum of conjugates (under~$\J_\pi$)
of the inflation of an irreducible cuspidal representation~$\tau_\pi$
of the reductive quotient~$\Gg^\so_{N_1,N_2}$. By~\cite{MP,Vigneras},
the standard maximal parahoric subgroup~$\J^\so_{\pi}$ is uniquely
determined by~$\pi$, so that~$N_1,N_2$ here are determined by~$\pi$,
and the representation~$\tau_\pi$ is determined
up to conjugacy by an element of~$\Gg_{N_1,N_2}$. Since the
group~$\Gg_{N_1,N_2}^\so$ decomposes as~$\Gg^{(1)}_{N_1}\times\Gg^{(2)}_{N_2}$, 
we can also write~$\tau_\pi=\tau_\pi^{(1)}\otimes\tau_\pi^{(2)}$,
with~$\tau_\pi^{(i)}$ an irreducible cuspidal representation of~$\Gg^{(i)}_{N_i}$.

The irreducible cuspidal representations~$\tau$ of the groups~$\Gg^\so_{N_1,N_2}$ were
classified by Lusztig~\cite{L77,Ldisc}, in terms of semisimple
elements~$s$ of the dual group and unipotent irreducible cuspidal representations
of the centralizer of~$s$, generalizing the classification of
Green. (The only unipotent irreducible cuspidal representation of~$\GL_n(k_\F)$ is
the trivial representation of~$\GL_1(k_\F)$.) We will recall this
later, in Section~\ref{S.computation}, when we require it.

\section{Reducibility of parabolic induction}
\label{S.reducibility}

In this section, we recall some basic results, in particular due to
Silberger, on reducibility of parabolic induction. We continue with
the same notation, so that~$\G=\U(\V)^\so$ is our classical group. We
recall that we have the group~$\G_n=\U(\V_n)^\so$, with Levi
subgroup~$\M_n\simeq\GL_n(\F)\times\G$ (with the isomorphism
determined by the chosen Witt basis) and standard parabolic
subgroup~$\P_n=\M_n\N_n$. Let~$\rho\in\Aa_n(\F)$ and~$\pi\in\Aa(\G)$,
so that we can consider~$\rho\otimes\pi$ as a representation
of~$\M_n$. 

We are interested in the (ir)reducibility of the normalized
parabolically induced representation
\[
I(\rho,\pi,s) = \Ind_{\P_n}^{\G_n}\rho|\det(\cdot)|_\F^s\otimes\pi,
\]
for~$s\in\CC$, where~$|\cdot|_\F$ is the normalized absolute value
on~$\F$ (with image~$q^\ZZ$) and~$\det$ is the determinant
on~$\GL_n(\F)$. We note that replacing~$\rho$ by an unramified twist
just has the effect of translating the parameter~$s$; that is
\[
I(\rho|\det(\cdot)|_\F^{s_0},\pi,s)=I(\rho,\pi,s+s_0).
\]
Thus we lose no information if we replace our base-point~$\rho$ with
any unramified twist. We have the following fundamental result of Silberger:
the first part comes from~\cite[Corollaries~5.4.2.2--3]{Silbook} and
the second from~\cite[Theorem~1.6]{Sil}.

\begin{theorem}\label{thm:reducibility}
\begin{enumerate}
\item If~$I(\rho,\pi,s)$ is reducible for some~$s\in\RR$, then there
 exists~$s_0\in\RR$ such that~$\rho|\det(\cdot)|_\F^{s_0}$ is self-dual.
\item If~$\rho$ is self-dual and~$I(\rho,\pi,s)$ is reducible for
 some~$s\in\RR$, then there is a (unique) real number~$s_\pi(\rho)\ge
 0$ such that, for~$s\in\RR$,
\[
I(\rho,\pi,s)\text{ is reducible if and only if }s=\pm s_\pi(\rho).
\]
\end{enumerate}
\end{theorem}

\begin{remark}\label{rmk:nored}
In the situation of Theorem~\ref{thm:reducibility}, it is almost true
that, if~$\rho$ is self-dual then~$I(\rho,\pi,s)$ is reducible for
some~$s\in\RR$. The only exception comes from even special orthogonal
groups, where we have an extra subtlety (see~\cite{Jantzen} for more
details): if~$\pi$ is an irreducible cuspidal representation of~$\G$ which
is \emph{not} normalized by the full orthogonal group~$\G^+$ and~$n=1$
(so that~$\rho$ is a trivial or quadratic character of~$\F^\times$)
then~$I(\rho,\pi,s)$ is irreducible for all~$s\in\RR$. On the other
hand, in this situation, putting~$\pi^+=\Ind_{\G}^{\G^+}\pi$, which is
an irreducible cuspidal representation of~$\G^+$,
then~$I(\rho,\pi^+,s)$ does have reducibility, at~$s=0$.
\end{remark}

From Theorem~\ref{thm:reducibility}, for a fixed~$\pi\in\Aa(\G)$, we get a map
\[
s_\pi:\Aa^\s(\F)\to\RR_{\ge 0},
\]
where we define~$s_\pi(\rho)=0$ if~$I(\rho,\pi,s)$ is irreducible for
all~$s\in\RR$. 
Part of the well-known ``Basic Assumption'' made in~\cite{MT} is that
the image of this map is in fact in~$\frac 12\ZZ$ (indeed, this is now
known in many cases -- at least when~$\G$ is quasi-split -- through
the work of Arthur, M\oe glin, Waldspurger). We will prove here
(independently) that this is indeed the case for depth zero
representations~$\pi$.

Silberger's results in fact give a little more than stated here, since
we have stated them only for \emph{real} values of~$s$. Indeed,
if~$\rho\in\Aa^\s(\F)$ then there are (up to equivalence) exactly
\emph{two} unramified twists of~$\rho$ which are self-dual:~$\rho$ and
another one~$\rho'$. If, moreover,~$\rho$ is a depth zero
representation then this second representation is easy to describe: it
is~$\rho':=\rho|\det(\cdot)|_\F^{\pi\ii/n\log q}$. Thus Silberger's
result in fact gives a qualitative description of all complex~$s$ for
which~$I(\rho,\pi,s)$ is reducible.

In general, we will here only be able to compute the pair of
numbers~$\{s_\pi(\rho),s_\pi(\rho')\}$, rather than distinguishing
them individually. However, this is sufficient to prove the
equality~\eqref{eqn:expected}.

\section{Covers and Hecke algebras}
\label{S.covers}

The theory of types and covers was developed by Bushnell--Kutzko to
give a strategy and framework to describe the structure of the
category of smooth representations of a connected reductive
group. Here we are interested only in a rather special case (in
particular, we have only maximal proper Levi subgroups and depth zero
representations for classical groups) so we do not give definitions
and results in their full generality. In particular, we are
specializing to depth zero the results of~\cite[\S3.2]{Blb}.

We continue in the notation of the previous section but restrict to
depth zero. Thus we have~$\rho\in\Aa^\s_{[0]}(\F)$, a representation
of~$\GL_n(\F)$, and~$\pi\in\Aa_{[0]}(\G)$, giving us a
representation~$\rho\otimes\pi$ of the Levi
subgroup~$\M_n\simeq\GL_n(\F)\times\G$ of~$\G_n$.
We write~$\rho= \cind_{\BJ_\rho}^{\GL_n(\F)} \La_\rho$
and~$\pi = \cind_{\J_\pi}^{\G}\l_\pi$, as in Section~\ref{S.cuspidals}, and
use all the associated notation from there. We write~$\P_n=\M_n\N_n$ and denote 
by~$\P^-_m=\M_n\N^-_n$ the opposite parabolic subgroup.

We put~$\J_\M=\J_\rho\times\J_\pi$, a compact open subgroup of~$\M_n$,
and~$\l_\M=\l_\rho\otimes\l_\pi$, an irreducible representation
of~$\J_\M$. From~\cite[Theorem~1.1]{MiS}, there is
a~\emph{cover}~$(\J,\l)$ of~$(\J_\M,\l_\M)$, that is
\begin{itemize}
\item $\J$ is a compact open subgroup of~$\G_n$ which has an Iwahori decomposition with respect to~$(\M_n,\P_n)$ and such that~$\J\cap\M_n=\J_\M$;
\item $\l$ is an irreducible representation of~$\J$ whose restriction to~$\J_\M$ is~$\l_\M$ and whose restriction to~$\J\cap\N^{\pm}_n$ is a multiple of the trivial representation;
\item the Hecke algebra~$\Hh(\G_n,\l)$ contains an invertible element whose support is the~$(\J,\J)$-double coset of a strongly positive element of the centre of~$\M_n$.
\end{itemize}
Moreover, we have a description of the Hecke algebra~$\Hh(\G_n,\l)$
given by~\cite[Theorem~1.2]{MiS}:
\begin{itemize}
\item[(i)] If there is some~$g\in\G_n\setminus\M_n$ which
 normalizes~$\M_n$ and such that the conjugate by~$g$
 of~$\rho\otimes\pi$ is equivalent to an unramified twist
 of~$\rho\otimes\pi$, then~$\Hh(\G_n,\l)$ is a generic Hecke algebra
 on an infinite dihedral group; that is, it is generated
 by~$T_1,T_2$, each supported on a single~$(\J,\J)$-double coset,
 with relations
\[
(T_i-q^{f_i})(T_i+1)=0,
\]
for some half-integers~$f_i\ge 0$. Moreover, there is a recipe to
compute the~$f_i$, which we revisit in Section~\ref{S.reduction}
below.
\item[(ii)] Otherwise,~$\Hh(\G_n,\l)$ is abelian, isomorphic
 to~$\CC\left[Z^{\pm 1}\right]$.
\end{itemize}
In the second case, the induced representation~$I(\rho,\pi,s)$ is
irreducible for any~$s\in\CC$ so we restrict our interest to the first
case. Since~$\rho$ is self-dual, the condition in (i) is always
satisfied, unless~$\G$ is an even-dimensional special orthogonal
group, the representation~$\pi$ is~\emph{not} normalized by the full
orthogonal group~$\G^+$, and~$n=1$. (See Remark~\ref{rmk:nored}.)

\medskip

We write~$\Rr^{[\rho,\pi]}(\M_n)$ for the full subcategory of (smooth complex)
representations of~$\M_n$ all of whose irreducible subquotients are
unramified twists of~$\rho\otimes\pi$, and~$\Rr^{[\rho,\pi]}(\G_n)$
for the full subcategory of representations of~$\G_n$ all of whose
irreducible subquotients have supercuspidal support an unramified
twist of~$\rho\otimes\pi$. Then, since~$(\J,\l)$ is a cover
of~$(\J_\M,\l_\M)$ which is a type, we have a normalized embedding of
Hecke algebras~$t:\Hh(\M_n,\l_\M)\hookrightarrow\Hh(\G_n,\l)$ giving
us a commutative diagram
\[
\xymatrix{
\Rr^{[\rho,\pi]}(\G_n) \ar[r]^{\!\!\!\hh\ \ }&\Mod\Hh(\G_n,\l) \\
\Rr^{[\rho,\pi]}(\M_n) \ar[u]^{\Ind_{\P_n}^{\G_n}}
\ar[r]^{\!\!\!\hh_\M\ \ } &\Mod\Hh(\M_n,\l_\M)\ar[u]_{t_*}
}
\]
Here, the functor~$t_*$ maps a right~$\Hh(\M_n,\l_\M)$-module~$X$
to~$\Hom_{\Hh(\M_n,\l_\M)}(\Hh(\G_n,\l),X)$, where
the~$\Hh(\M_n,\l_\M)$-module structure on~$\Hh(\G_n,\l)$ is given
by~$t$. The horizontal arrows are equivalences of categories: the
functor~$\hh_\M$ is given by~$\xi\mapsto\Hom_{\J_\M}(\l_M,\xi)$, and
similarly for~$\hh$.

The Hecke algebra~$\Hh(\M_n,\l_\M)$ is isomorphic to~$\CC[Z^{\pm 1}]$,
where~$Z$ is supported on~$\z\J_\M$ and~$\z$ is the element of the
centre of~$\M_n$ which acts on~$\V_n=n\H^-\oplus\V\oplus n\H^+$ as~$1$
on~$\V$ and as~$\varpi_\F$ on~$n\H^+$. The
element~$T_2T_1\in\Hh(\G_n,\l)$ is supported on the double
coset~$\J\z\J$ and, since~$t(Z)$ is supported on the same double
coset, we may (and do) normalize~$Z$ so that~$t(Z)=T_2T_1$.

Now, from~\cite[Proposition~3.12]{Blb}, we get that,
if~$I(\rho,\pi,s)$ is reducible, then the real part of~$s$ belongs to
the set
\begin{equation}\label{eqn:redpoints}
\left\{\pm\frac{(f_1\pm f_2)}{2n}\right\}.
\end{equation}
We will see these values are always half-integers. Thus, in the
notation of Section~\ref{S.reducibility}, we have
\[
\{s_\pi(\rho),s_\pi(\rho')\}=\left\{\frac{|f_1\pm f_2|}{2n}\right\},
\]
where~$|\cdot|$ is the usual (real) absolute value, and we recall
that~$\rho'$ is the unique self-dual unramified twist of~$\rho$ which
is not equivalent to~$\rho$.

\section{Reduction to the finite case}
\label{S.reduction}

We now describe how to relate the parameters~$f_i$ of the
previous section to a problem in the representation theory of finite
reductive groups, and rephrase the equality~\eqref{eqn:expected} in
these terms. The recipe for computing these parameters is described
in~\cite[\S6.3]{MiS}, which is considerably simplified by only
treating depth zero representations; in particular, the
character~$\chi$ of \emph{loc.~cit.} is trivial. Thus the content of
this section is all proved in \emph{loc.~cit.} and here we just
explicate it in our special case.

We recall that~$\pi=\cind_{\J_\pi}^{\G}\l_\pi$, where~$\J_\pi$ is the
normalizer of a standard maximal parahoric subgroup~$\J^\so_{N_1,N_2}$,
with~$N_1,N_2$ uniquely determined, and~$\l_\pi|_{\J^\so_{N_1,N_2}}$
contains an irreducible representation~$\l_\pi^\so$ inflated from an irreducible cuspidal
representation~$\tau_\pi\simeq\tau_\pi^{(1)}\otimes\tau_\pi^{(2)}$ of the
(connected) reductive
quotient~$\Gg_{N_1,N_2}^\so\simeq\Gg_{N_1}^{(1)}\times\Gg_{N_2}^{(2)}$. More explicitly, as in
Section~\ref{S.notation}, we write
\[
\Gg_{N_i}^{(i)}=\U(N_i\ov\H\oplus\ov\V^\an_{(i)})^\so,\qquad
\text{for }i=1,2. 
\]
We will also write~$\Gg_{N_i}^{(i)+}$ for the full finite classical group of
which~$\Gg_{N_i}^{(i)}$ is the connected component. We will need to
distinguish the cases when~$\tau_\pi^{(i)}$ is normalized by~$\Gg_{N_i}^{(i)+}$
from those where it is not.

By construction, the group~$\J$ in~$\G_n$ has reductive quotient
isomorphic to~$\GL_n(k_\F)\times\Gg_{N_1,N_2}$, and we denote by~$\J^\so$ the
inverse image of its connected component. The parahoric
subgroup~$\J^\so$ is contained in precisely two maximal
compact subgroups of~$\G_n$, namely the standard maximal
compact~$\J_2:=\J_{N_1,N_2+n}$ and a maximal compact group~$\J_1$
conjugate to~$\J_{N_1+n,N_2}$. 
More precisely, both~$\J_1,\J_2$ would be standard with respect to the ordered
Witt basis
\[\begin{split}
&\ee_{N}^-,\ldots,\ee_{N_1+1}^-,
\ee_{N+n}^-,\ldots\ee_{N+1}^-,
\ee_{N_1}^-,\ldots,\ee_1^-,
\ee_1^\an,\ldots,\ee_{N^\an}^\an,\\
&\hspace{4.4cm}\ee_{1}^+,\ldots,\ee_{N_1}^+,
\ee_{N+1}^+,\ldots,\ee_{N+n}^+,
\ee_{N_1+1}^+,\ldots,\ee_{N}^+.
\end{split}\]
These maximal compact subgroups have
reductive quotients~$\Gg_1,\Gg_2$ isomorphic to~$\Gg_{N_1+n,N_2}$
and~$\Gg_{N_1,N_2+n}$ respectively, and the image of~$\J^\so$ in each
of these is a parabolic subgroup with Levi component isomorphic
to~$\GL_n(k_\F)\times\Gg_{N_1,N_2}^\so$. More explicitly,
the~$\Gg_i$ have connected components 
\begin{eqnarray*}
&&\Gg_1^{\so}\simeq \U((N_1+n)\ov\H\oplus\ov\V^\an_{(1)})^\so \times
\Gg_{N_2}^{(2)} \supseteq 
 \(\GL_n(k_\F)\times\Gg_{N_1}^{(1)} \) \times \Gg_{N_2}^{(2)}, \\
&&\Gg_2^{\so}\simeq \Gg_{N_1}^{(1)} \times
\U((N_2+n)\ov\H\oplus\ov\V^\an_{(2)})^\so \supseteq 
 \Gg_{N_1}^{(1)} \times \(\GL_n(k_\F)\times \Gg_{N_2}^{(2)}\).
\end{eqnarray*}
Moreover, we have certain Weyl group elements~$s_i\in\J_i$, defined
in~\cite[\S5.6]{MiS}. (See also~\cite[\S6.2]{S5} where, in most cases,
they are denoted~$s^\varpi_1,s_1$ respectively, though there is an
added complication when~$n=1$ and~$\G$ is a special orthogonal group,
explained in~\cite[\S5.6]{MiS}.) More explicitly, both~$s_1,s_2$
exchange (up to scalars) the vectors~$\ee_{N+j}^+$ and~$\ee_{N+j}^-$,
for~$1\le j\le n$ and preserve the subspace~$\V$ of~$\V_n$.

Now let~$(\J,\l)$ be the cover
of~$(\J_\rho\times\J_\pi,\l_\rho\otimes,\l_\pi)$ from
Section~\ref{S.covers}. The proof of the existence of this cover,
in~\cite{MiS}, goes via first constructing a cover~$(\J^\so,\l^\so)$
of~$(\J_\rho\times\J^\so_\pi,\l_\rho\otimes\l^\so_\pi)$. The Hecke
algebras~$\Hh(\G_n,\l^\so)$ and~$\Hh(\G_n,\l)$ are not in general
isomorphic, but they are closely related, as we now describe.

The Hecke algebra~$\Hh(\G_n,\l)$ is generated by two
elements~$T_1,T_2$, with~$T_i$ supported on~$\J s_i\J$ and satisfying
a quadratic relation~$(T_i-q^{f_i})(T_i+1)=0$. Then:
\begin{enumerate}
\item If~$n=1$, and either~$\Gg_{N_i}^{(i)+}$ is the trivial
 orthogonal group (\ie the orthogonal group on a trivial space) or
the irreducible cuspidal representation~$\tau_\pi^{(i)}$ of~$\Gg_{N_i}^{(i)}$ is \emph{not}
normalized by~$\Gg_{N_i}^{(i)+}$, then we always have~$f_i=0$; that is,~$T_i^2=1$.
Note that our assumption that we have reducibility implies that this
can be the case for at most one value of~$i$, and this is never
the case if~$\G$ is either symplectic or unramified unitary.
\item Otherwise, there is a corresponding
 element~$T_i^\so\in\Hh(\G_n,\l^\so)$ with support~$\J^\so s_i\J^\so$
 and satisfying the same quadratic relation as~$T_i$ (with the same
 parameter~$f_i$).
\end{enumerate}
It remains to describe the parameter~$f_i$ in the latter case and we
assume from now on that we are in that situation.

By inflation, we have a support-preserving algebra injection
\[
\Hh(\Gg_i^{\so},\tau_\rho\otimes\tau_\pi^\so)\hookrightarrow 
\Hh(\G_n,\l^\so),
\]
and~$T_i^\so$ is in the image of this map. We also have isomorphisms
\[
\Hh(\Gg_i^{\so},\tau_\rho\otimes\tau_\pi^\so) \simeq
\Hh(\U((N_i+n)\ov\H\oplus\ov\V^\an_{(i)})^\so,\tau_\rho\otimes\tau_\pi^{(i)}),
\]
and the algebra on the right is described (at least in the case that
the ambient finite classical group has connected centre)
in~\cite[Theorem~8.6]{L}: they are two-dimensional, generated by
an element satisfying the same quadratic relation, with~$q^{f_i}$ the
quotient of the dimensions of the two irreducible factors of the
representation parabolically induced
from~$\tau_\rho\otimes\tau_\pi^{(i)}$. Moreover, one can compute this by
using Lusztig's Jordan decomposition of characters, as we describe
in the next section.

\section{Computation of parameters}
\label{S.computation}

In this section, we undertake the computation of the parameters in the
finite Hecke algebras from above. When the finite reductive group
arising has connected centre, this can more-or-less be read off from
the Jordan decomposition of characters and the case of unipotent
irreducible cuspidals, for which there are tables in~\cite{LCBMS}. In general, one
must first embed the group into one with connected centre, and then
make the comparison. A special case is already carried out
in~\cite{KM}, where they look at the Hecke algebra coming from
inducing a self-dual irreducible cuspidal representation of the Siegel parabolic
of a classical group. 

In order to fit with the usual notations for finite reductive groups,
the notation in this section is independent of that in the rest of the
paper. We will not recall here the definitions of~\emph{geometric} and~\emph{rational Lusztig series}, both of which give partitions of the set of irreducible representations of a connected finite reductive group coming from Deligne--Lusztig induction; we refer the reader instead, for example, to~\cite{CE} or~\cite{DM}.

\subsection{Self-dual polynomials}
\label{SS.polynomials}

We begin with a brief section on irreducible self-dual polynomials
over finite fields, since these will be used to parametrize the
irreducible cuspidal representations of our finite reductive
groups. We fix~$\FF_q$ a finite field of odd cardinality~$q$ and
let~$\FF_{q_\so}$ be a subfield of index at most~$2$. We denote
by~$\s$ the automorphism generating~$\Gal(\FF_q/\FF_{q_\so})$, and use
the same notation for the induced automorphism of the polynomial
ring~$\FF_q[X]$, obtained by applying~$\s$ to all coefficients. 

For~$P\in\FF_q[X]$, we put
\[
P^\s(X)\ :=\ \s(P(0))^{-1}X^{\deg(P)}\s(P)(1/X).
\]
We say that a monic polynomial~$P\in\FF_q[X]$ is
\emph{$\FF_q/\FF_{q_\so}$-self-dual} if~$P=P^\s$; thus~$P$ is
$\FF_q/\FF_{q_\so}$-self-dual if and only if:
\begin{enumerate}
\item when~$\FF_q=\FF_{q_\so}$, for each root~$\zeta$ of~$P$
(in some algebraic closure of~$\FF_q$), the element~$\zeta^{-1}$
is also a root of~$P$;
\item when~$\FF_q\ne\FF_{q_\so}$, for each root~$\zeta$ of~$P$, the
 element~$\zeta^{-q_\so}$ is also a root of~$P$.
\end{enumerate}
When~$\FF_q=\FF_{q_\so}$, we will just speak of \emph{self-dual}
polynomials; these might more often elsewhere be called \emph{reciprocal}.

If we now restrict to \emph{irreducible} $\FF_q/\FF_{q_\so}$-self-dual
monic polynomials~$P$, the possibilities are somewhat constrained:
\begin{enumerate}
\item when~$\FF_q=\FF_{q_\so}$, either~$P(X)=X\pm 1$ or else~$\deg(P)$
 is even;
\item when~$\FF_q\ne\FF_{q_\so}$, we must have that~$\deg(P)$ is odd.
\end{enumerate}

\subsection{Connected centre}
\label{SS.connected}

We now turn to the problem at hand, beginning in the case of a group
with connected centre, so that the
centralizer of any semisimple element of the dual group is
connected. Let~$\Gg$ be a connected reductive group of \emph{classical type},
over a finite field~$\FF_q$ of odd characteristic~$p$, \emph{with
connected centre} and with Frobenius map~$\Ff$. By classical type here
we mean that~$\Gg$ is one of:
\begin{enumerate}
\item an odd-dimensional special orthogonal group~$\SO_{2N+1}$;
\item a group of symplectic similitudes~$\GSp_{2N}$;
\item a group of orthogonal similitudes~$\GSO_{2N}^+$ or~$\GSO_{2N}^-$,
 of Witt index~$N$,~$N-1$ respectively. (Note that we mean here
 that~$\GSO_{2N}^{\pm}$ is the connected component of the full group of
 orthogonal similitudes~$\GO_{2N}^{\pm}$.) In this case we do
 \emph{not} allow the group~$\GSO_2^+$.
\end{enumerate}
In each case, the Frobenius map~$\Ff$ is the standard one.
We denote by~$\Gg^*$ the dual group and write~$\Ff$ again for the
(dual) Frobenius on it. The dual group acts naturally on
an~$\overline\FF_q$-vector space~$\Vv$, with an~$\FF_q$-structure and a
form, of dimension~$2N,2N+1,2N$ respectively in the three cases above. 
In case (iii), we say that~$\Vv$ is of type~$+1$ if it has Witt index~$N$, 
and of type~$-1$ otherwise; we say that the zero space has type~$+1$.

Write~$\Ee(\Gg^\Ff)$ for the set of equivalence classes of irreducible
(complex) representations of~$\Gg^\Ff$. Then (see for example~\cite[\S7.6]{L77})
there is a partition into \emph{geometric Lusztig series}
\[
\Ee(\Gg^\Ff)\ =\ \bigcup_s \Ee(\Gg^\Ff,s),
\]
where~$s$ runs over the conjugacy classes of semisimple elements
of~$\Gg^{*,\Ff}$. (Note that rational and geometric conjugacy classes
coincide as the centre of~$\Gg$ is connected.) The partition is given
as follows: for any~$\Ff$-stable maximal torus~$\Tt$ of~$\Gg^*$
containing~$s$, we have the Deligne--Lusztig
representation~$\R_\Tt^\Gg s$; then an irreducible
representation~$\pi$ of~$\Gg^\Ff$ lies in~$\Ee(\Gg^\Ff,s)$ if and only
if there is such a torus~$\Tt$ with
\[
\la \pi,\R_\Tt^\Gg s\ra\ \ne\ 0.
\]
(Here,~$\la\cdot,\cdot\ra$ denotes the natural~$\Gg$-invariant inner
product on class functions, and we identify (equivalence classes of)
representations with their characters.) 

Given a semisimple element~$s\in\Gg^{*,\Ff}$,
the centralizer~$\Gg^*_s$ is a connected reductive group of the same
rank as~$\Gg$, though in general it is not a Levi subgroup.
Then the Jordan decomposition of
characters~\cite[Corollary~7.10]{L77} (see
also~\cite[Theorem~15.8]{CE}) gives a bijection
\[
\psi_s^\Gg:\Ee(\Gg^\Ff,s)\to\Ee(\Gg^{*,\Ff}_s,1) 
\]
with the following properties (see~\cite[\S7.8]{L77}):
\begin{itemize}
\item for any irreducible representation~$\pi$ in~$\Ee(\Gg^\Ff,s)$ and
 any~$\Ff$-stable maximal torus~$\Tt$ containing~$s$,
\begin{equation}\label{eqn:char}
\e_{\Gg}\la \pi,\R_\Tt^\Gg s\ra = \e_{\Gg_s}\la \psi_s^G(\pi),\R_\Tt^{\Gg_s^*}1\ra,
\end{equation}
where~$\e_{\Gg}=(-1)^{\FF_q\mrank(\Gg)}$ (see~\cite[Definition~8.3]{DM});
\item there is a uniform constant~$c_s$ such that
\[
\dim\pi\ =\ c_s\dim \psi_s^\Gg(\pi),
\]
for all~$\pi$ in~$\Ee(\Gg^\Ff,s)$; explicitly, writing~$d_{p'}$ for
the maximal divisor of~$d$ coprime to~$p$, for any positive
integer~$d$, we have~$c_s=|\Gg^{*,\Ff}|_{p'}|\Gg_s^{*,\Ff}|^{-1}_{p'}$;
\item if the identity components of the centres of~$\Gg^*$
and~$\Gg_s^*$ have the same~$\Ff_q$-rank then~$\pi$
in~$\Ee(\Gg^\Ff,s)$ is cuspidal if and only if~$\psi_s^\Gg(\pi)$ is
cuspidal; otherwise, no representation in~$\Ee(\Gg^\Ff,s)$ is
cuspidal.
\end{itemize}
(We remark that all the results so far extend to the case of
disconnected centre once we replace the geometric Lusztig series with
the rational series -- see below in section~\ref{SS.disconnected} for the notion
of rational Lusztig series, which coincides with the geometric one for groups 
with connected cntre.)

Thus we get a classification of the irreducible cuspidal
representations of~$\Gg^\Ff$ from a classification of
pairs~$(s,\tau)$, with~$s$ a semisimple element of~$\Gg^{*,\Ff}$ (up
to conjugacy) such that the identity components of the centres
of~$\Gg^*$ and~$\Gg_s^*$ have the same~$\Ff_q$-rank, and~$\tau$ an
irreducible cuspidal unipotent representation of~$\Gg_s^{*,\Ff}$ (up to
equivalence). Lusztig classified the irreducible cuspidal unipotent
representations of classical groups -- in particular, there is at most
one irreducible cuspidal unipotent representation for each such group -- and we find
(see~\cite[p172]{L77}) the following. 

For~$s\in\Gg^{*,\Ff}$ semisimple, we denote by~$P_s\in\FF_q[X]$ its
characteristic polynomial (as an automorphism of the space~$\Vv$ on
which~$\Gg^*$ acts naturally), by~$\Vv_+$ the~$(+1)$-eigenspace and by~$\Vv_-$
the~$(-1)$-eigenspace. Then there is
a bijection between the (equivalence classes of) irreducible cuspidal
representations of~$\Gg^\Ff$ and the set of conjugacy classes of
semisimple elements~$s\in\Gg^{*,\Ff}$ such that
\[
P_s(X)\ =\ \prod_{P} P(X)^{a_P},
\]
where the product runs over all irreducible self-dual monic
polynomials over~$\FF_q$ and the integers~$a_p$ satisfy:

\def\theequation{(\arabic{section}.\arabic{equation})}
\begin{notitle}\label{eqn:notitle}
\begin{itemize}
\item[{$\bullet$\qquad\quad\,}] \ \hskip-1.2cm$\displaystyle\sum_P a_P\deg(P) = \dim\Vv$;
\item for~$P(X)\ne (X\pm 1)$, we have~$a_P=\tfrac
12(m_P^2+m_P^{})$, for some integer~$m_P^{}\ge 0$;
\item writing~$a_+:=a_{(X-1)}$ and~$a_-:=a_{(X+1)}$, there are
 integers~$m_+^{},m_-^{}\ge 0$ such that
\begin{enumerate}
\item if~$\Gg=\SO_{2N+1}$ then~$a_+=2(m_+^2+m_+^{})$
 and~$a_-=2(m_-^2+m_-^{})$,
\item if~$\Gg=\GSp_{2N}$ then~$a_+=2(m_+^2+m_+^{})+1$
 and~$a_-=2m_-^2$,
\item if~$\Gg=\GSO_{2N}^{\pm}$ then~$a_+=2m_+^2$ and~$a_-=2m_-^2$,
\end{enumerate}
where, in case~(iii),~$\Vv_{\pm}$ is an even-dimensional orthogonal space of
type~$(-1)^{m_\pm}$, and the same in case~(ii) for~$\Vv_-$ only.
\end{itemize}
\end{notitle}
\def\theequation{\arabic{section}.\arabic{equation}}

Let~$s$ be a semisimple element of~$\Gg^{*,\Ff}$ and suppose that we
have an~$\Ff$-stable Levi subgroup~$\Ll^*$ contained in
an~$\Ff$-stable parabolic subgroup~$\Pp^*$ of~$\Gg^*$ such
that~$s\in\Ll^*$. Correspondingly, we have an~$\Ff$-stable
Levi subgroup~$\Ll$ contained in an~$\Ff$-stable parabolic
subgroup~$\Pp$ of~$\Gg$. Then~$\Ll^*_s$ is an~$\Ff$-stable rational
Levi subgroup of~$\Gg^*_s$ (though it need not be a \emph{proper} Levi
subgroup), contained in an~$\Ff$-stable parabolic~$\Pp^*_s$. Then we
have a diagram
\[
\xymatrix{
\ZZ\Ee(\Gg^\Ff,s)\ar[rr]^{\psi_s^\Gg}&&\ZZ\Ee(\Gg_s^{*,\Ff},1)\\
\Ee(\Ll^\Ff,s)\ar[rr]^{\psi^\Ll_s}\ar[u]_{\Ind_{\Ll,\Pp}^{\Gg}}&&
\Ee(\Ll_s^{*,\Ff},1)\ar[u]_{\Ind_{\Ll_s^{*},\Pp_s^*}^{\Gg_s^{*}}}
}
\]
(The vertical arrows here are parabolic induction -- \ie
Harish-Chandra induction -- and we have abbreviated
from~$\Ind_{\Ll^\Ff,\Pp^\Ff}^{\Gg^\Ff}$ since the notation is already heavy.)
This diagram commutes (this is a result of Shoji, which can be extracted from the appendix to~\cite{FS}); in the cases that interest us here it can be seen fairly directly:
\begin{itemize}
\item If~$\Gg_s^*\subseteq\Ll^*$ then the vertical arrows preserve
irreducibility (this is a special case of~\cite[(7.9.1)]{L77}) and
the diagram commutes by~\eqref{eqn:char}. In fact, this also
generalizes to the case where the parabolic~$\Pp$ is not~$\Ff$-stable,
replacing~$\Ind_{\Ll,\Pp}^{\Gg}$ by the Deligne--Lusztig
map~$\e_\Gg\e_\Ll\R_{\Ll\subset\Pp}^\Gg$; indeed, in that case the diagram
commutes~\emph{by definition} of the map~$\psi_s^\Gg$ (see the proof
of~\cite[Proposition~7.9]{L77}).
\item Suppose~$\Ll$ is a \emph{maximal} proper Levi subgroup
and~$\tau\in\Ee(\Ll,s)$ is cuspidal. Then~$\Nn_\Gg(\Ll)/\Ll$ has
order~$1$ or~$2$. We are interested in the case where~$\Ind_\Ll^\Gg\tau$ is
reducible; equivalently,~$\Nn_\Gg(\Ll)/\Ll$ has order~$2$ and, writing~$w$ for a representative of the nontrivial coset,~$w$ normalizes~$\tau$. In this case the
induced representation decomposes as
\[
\Ind_{\Ll,\Pp}^\Gg\tau\ =\ \pi_1\oplus\pi_2, \qquad\dim(\pi_1)>\dim(\pi_2),
\]
(the inequality is strict by~\cite[Theorem~8.6]{L})
and~$\End_{\Gg^\Ff}(\Ind_{\Ll,\Pp}^{\Gg}\tau)$ is a two-dimensional
algebra with a quadratic generator~$T$ satisfying a relation of the
form
\[
(T+1)(T-q^{f_\tau})\ =\ 0,\qquad q^{f_\tau}=\frac{\dim(\pi_1)}{\dim(\pi_2)}.
\]
Moreover, the same is true
for~$\Ind_{\Ll_s^*,\Pp_s^*}^{\Gg_s^*}\psi_s^\Ll(\tau)$ and the recipe given
in~\cite[\S8]{L} (see~\emph{op.\ cit.\/} Theorem~8.6) to
calculate~$f_\tau$ depends on the computation of certain Weyl groups
which are the same for both induced representations. (Indeed, this
matching is the idea behind the inductive proof of the Jordan
decomposition of characters.) On the side of the centralizer, we have
unipotent representations and the parameter~$q^{f_\tau}$ can be read
from the tables in~\cite{LCBMS}. In the special case that~$s^2=1$ (to
which one could reduce) one can also read off the parameter
from~\cite[Proposition~8.3]{L}.
\end{itemize}

\subsection{Disconnected centre}
\label{SS.disconnected}

We now consider the case which really interests us here. So we suppose
that~$\Gg$ is a \emph{classical group} over a finite field~$\FF_q$ of
odd characteristic~$p$, with Frobenius map~$\Ff$. By classical group here,
we mean that~$\Gg$ is one of:
\begin{enumerate}
\item an odd-dimensional special orthogonal group~$\SO_{2N+1}$;
\item a symplectic group~$\Sp_{2N}$;
\item an even-dimensional special orthogonal group~$\SO_{2N}^+$ or~$\SO_{2N}^-$,
 of Witt index~$N$,~$N-1$ respectively, where again we do not allow the group~$\SO_2^+$.
\end{enumerate}
In each case, the Frobenius map~$\Ff$ is again the standard
one. Case~(i) has already been treated above, so we only
consider cases~(ii),(iii) here.

In each case, we embed~$\Gg$ in a group~$\tGg$ with connected centre of
the type considered in Section~\ref{SS.connected}. Then we get a
map~$\tGg^{*}\to\Gg^{*}$ which maps conjugacy classes of
semisimple elements in~$\tGg^{*,\Ff}$ to (rational)
$\Gg^{*,\Ff}$-conjugacy classes of semisimple elements
in~$\Gg^{*,\Ff}$. 

The geometric conjugacy class of a semisimple element~$s$
in~$\Gg^{*,\Ff}$ splits into two~$\Gg^{*,\Ff}$-conjugacy classes if
and only if its centralizer~$\Gg_s^*$ is disconnected, which happens
if and only both~$1$ and~$-1$ are eigenvalues of~$s$. 
Here we also have a partition (into \emph{rational} Lusztig series) of
the set of equivalence classes of irreducible representations of~$\Gg^\Ff$,
\[
\Ee(\Gg^\Ff)\ =\ \bigcup_s \Ee(\Gg^\Ff,s),
\]
where~$s$ runs over the~$\Gg^{*,\Ff}$-conjugacy classes of semisimple
elements of~$\Gg^{*,\Ff}$. Each geometric Lusztig series is the union
of at most two rational Lusztig series, corresponding to the rational
conjugacy classes in a geometric conjugacy class. Moreover, if~$\tilde
s\in\tGg^{*,\Ff}$ maps to~$s\in\Gg^{*,\Ff}$, then the rational
series~$\Ee(\Gg^\Ff,s)$ is precisely the set of irreducible components
of the restriction to~$\Gg^\Ff$ of the representations in the Lusztig
series~$\Ee(\tGg^\Ff,\tilde s)$ (see~\cite[Proposition~15.6]{CE}).

\medskip

We begin by considering the irreducible cuspidal representations
of~$\Gg^\Ff$. An irreducible representation of~$\Gg^\Ff$ is cuspidal if and
only if it is a component of the restriction of an irreducible
cuspidal representation of~$\tGg^\Ff$. In general, an irreducible cuspidal
representation of~$\tGg^\Ff$ will decompose as a sum of at
most two pieces on restriction to~$\Gg^\Ff$,
inequivalent but of the same dimension when there are two (since they
are conjugate by~$\tGg^\Ff$). Precisely what happens is essentially
determined by~\cite[Lemma~8.9]{L77}, which treats the 
special case of quadratic unipotent representations, as follows.

Let~$\tilde s\in\tGg^{*,\Ff}$ be such that its image~$s\in\Gg^{*,\Ff}$ is an
involution and such that~$\Ee(\tGg^\Ff,\tilde s)$ contains an
irreducible cuspidal representation~$\pi$; then~$\pi|_{\Gg^\Ff}$ is irreducible if
and only if~$s=\pm 1$. (Note that~$s=-1$ is in fact not possible
when~$\G=\Sp_{2n}$.)

This is enough to deal with the general case because of the following
(see~\cite[Theorem~8.27]{CE}). Suppose~$\Ll$ is an~$\Ff$-stable Levi
subgroup of~$\Gg$ and~$\Pp$ is a parabolic subgroup (not
necessarily~$\Ff$-stable) with Levi~$\Ll$, and let~$\Ll^*$,~$\Pp^*$ be
the corresponding subgroups of~$\Gg$. Suppose~$s\in\Gg^{*,\Ff}$ is a
semisimple element such
that~$\Gg^{*,\so}_s\Gg^{*,\Ff}_s\subseteq\Ll^*$,
where~$\Gg^{*,\so}$ denotes the connected component of~$\Gg^{*}$. Then
Deligne--Lusztig twisted induction~$\R_{\Ll\subset\Pp}^\Gg$ gives a bijection 
\[
\e_\Gg\e_\Ll\R_{\Ll\subset\Pp}^\Gg:\Ee(\Ll^\Ff,s)\to\Ee(\Gg^\Ff,s).
\]
Moreover (cf.~\cite[\S7.9]{L77}), there is a
constant~$c_{\Ll,\Gg}=\lvert \Gg^\Ff\rvert_{p'}\lvert \Ll^\Ff\rvert^{-1}_{p'}$
such that, for any~$\pi\in \Ee(\Ll^\Ff,s)$, 
\[
\dim\(\e_\Gg\e_\Ll\R_{\Ll\subset\Pp}^\Gg(\pi)\)\ =\ c_{\Ll,\Gg}\dim(\pi).
\]
Finally, this map respects cuspidality when the connected centres
of~$\Gg^*$ and~$\Ll^*$ have the same~$\FF_q$-rank: \ie in that
case,~$\pi\in\Ee(\Ll^\Ff,s)$ is cuspidal if and only
if~$\e_\Gg\e_\Ll\R_{\Ll\subset\Pp}^\Gg\pi$ in~$\Ee(\Gg^\Ff,s)$ is cuspidal.

Now, given a semisimple element~$\tilde s$ of~$\tGg^{*,\Ff}$, mapping
to~$s\in\Gg^{*,\Ff}$, such that~$\Ee(\tGg^\Ff,\tilde s)$ contains a
cuspidal representation, denote by~$\Vv=\bigoplus_{P}\Vv_P\oplus\Vv_0$
the decomposition of~$\Vv$ into the rational eigenspaces corresponding
to the irreducible monic self-dual polynomials of even degree,
with~$\Vv_0=\Vv_+\oplus\Vv_-$ the~$(\pm 1)$-eigenspace; then the
stabilizer~$\Ll^*$ of this decomposition is an~$\Ff$-stable Levi
subgroup containing the centralizer of~$s$, and the connected centres
of~$\Gg^*$ and~$\Ll^*$ have the same~$\FF_q$-rank. Thus we have a
bijection between the irreducible cuspidal representations
in~$\Ee(\Gg^\Ff,s)$ and in~$\Ee(\Ll^\Ff,s)$. Moreover, denoting
by~$\tLl,\tPp$ the inverse images of~$\Ll,\Pp$ in~$\tGg$ respectively,
we have (see~\cite[(15.5)]{CE}) a commutative diagram
\[
\xymatrix{
\Ee(\tLl^\Ff,\tilde s)\ar[rr]^{\R_{\tLl\subset\tPp}^\tGg}\ar[d]^{\Res^{\tLl}_{\Ll}}
&&\Ee(\tGg^{\Ff},\tilde s)\ar[d]^{\Res^{\tGg}_{\Gg}}\\
\ZZ\Ee(\Ll^\Ff,s)\ar[rr]^{\R_{\Ll\subset\Pp}^\Gg}
&&\ZZ\Ee(\Gg^{\Ff},s),
}
\]
where we have omitted the superscripts~$\Ff$ in the functors, for ease
of notation.

Putting this together, we get that, for general semisimple~$\tilde s$
of~$\tGg^{*,\Ff}$, mapping to~$s\in\Gg^{*,\Ff}$, and~$\pi$ an
irreducible cuspidal representation in~$\Ee(\tGg^\Ff,\tilde s)$, the
restriction~$\pi|_{\Gg^\Ff}$ is irreducible if and only if~$s$ acts
as~$\pm 1$ on~$\Vv_0$. Thus, in fact, the restriction remains
irreducible if and only if the centralizer~$\Gg_s^*$ is connected. We
have proved:

\begin{lemma}\label{lem:extend}
Let~$\pi\in\Ee(\Gg^\Ff,s)$ be an irreducible cuspidal representation
and let~$\tilde s\in\tGg^{*,\Ff}$ be a semisimple element mapping
to~$s$. The following are equivalent:
\begin{enumerate}
\item $\pi$ extends to an irreducible representation
 in~$\Ee(\tGg^\Ff,\tilde s)$;
\item the centralizer~$\Gg_s^*$ is connected;
\item at most one of~$\pm 1$ is an eigenvalue of~$s$.
\end{enumerate}
\end{lemma}

\medskip

Now we turn to the computation of the parameter. Thus, as in
Section~\ref{SS.connected}, we have~$s$ a semisimple element
of~$\Gg^{*,\Ff}$ and we suppose that~$\Ll^*$ is a maximal
proper~$\Ff$-stable Levi subgroup contained in an~$\Ff$-stable parabolic
subgroup~$\Pp^*$ of~$\Gg^*$ such that~$\Gg_s^*\subseteq\Ll^*$. 
Correspondingly, we have~$\Ff$-stable Levi and parabolic
subgroups~$\Ll,\Pp$ in~$\Gg$. We lift~$s$ to~$\tilde s\in\tGg^*$ and
likewise have lifts~$\tLl^*$ and~$\tPp^*$, with~$\tGg_{\tilde
 s}^*\subseteq\tLl^*$, and corresponding subgroups~$\tLl,\tPp$
in~$\tGg$ into which~$\Ll,\Pp$ respectively embed.

We are given a cuspidal representation~$\tau$ in~$\Ee(\Ll^\Ff,s)$, so
that~$\tau$ appears (with multiplicity one) in the restriction of a
cuspidal representation~$\tilde\tau$ in~$\Ee(\Ll^\Ff,\tilde s)$,
with~$\tilde s\in\tLl^{*,\Ff}$ mapping to~$s$. Let~$w$ be any
representative for the non-trivial element of~$\Nn_\Gg(\Ll)/\Ll$. We
will eventually be interested in the case where~$w$ normalizes~$\tau$
(so that the parabolically induced representation from~$\tau$ is
reducible) but, in that case, it is not immediately clear
whether or not~$w$ normalizes~$\tilde\tau$. We will need to know
exactly when this is the case but we can already say something about the
parameters in the Hecke algebras. For ease of notation, we again omit the
superscripts~$\Ff$ in the following.

\begin{lemma}\label{lem:paramtau}
\begin{enumerate}
\item\label{lem:paramtau.i}
Suppose~$w$ normalizes~$\tilde\tau$. Then~$w$ also normalizes~$\tau$
and we have an isomorphism of Hecke
algebras~$\End_\Gg(\Ind_{\Ll,\Pp}^\Gg\tau)\simeq\End_\tGg(\Ind_{\tLl,\tPp}^\tGg\tilde\tau)$.
\item\label{lem:paramtau.ii}
Suppose~$w$ does not normalize~$\tilde\tau$ but does
normalize~$\tau$. Then~$\End_\Gg(\Ind_{\Ll,\Pp}^\Gg\tau)$ has
generator~$T$ satisfying~$(T+1)(T-1)=0$.
\end{enumerate}
\end{lemma}

\begin{proof}
We write
\[
\Ind_{\Ll,\Pp}^\Gg\tau\ =\ \bigoplus_{i=1}^\ell\pi_i,
\qquad\text{where }
\ell=\begin{cases}
2,&\text{ if~$w$ normalizes~$\tau$,}\\
1,&\text{ otherwise.}
\end{cases}
\]
If~$\ell=2$ then we order the terms so that~$\dim(\pi_1)\ge\dim(\pi_2)$,
and~$\End_{\Gg}(\Ind_{\Ll,\Pp}^{\Gg}\tau)$ is a two-dimensional
algebra with a quadratic generator~$T$ satisfying a relation of the
form
\[
(T+1)(T-q^{f_\tau})\ =\ 0,\qquad q^{f_\tau}=\frac{\dim(\pi_1)}{\dim(\pi_2)}.
\]

In either case, from Mackey we also have
\begin{equation}\label{eqn:taumackey}
\Res_{\Gg}^{\tGg}\Ind_{\tLl,\tPp}^{\tGg}\tilde\tau
\ =\ \bigoplus_{\g\in\tLl/\Nn_{\tLl}(\tau)}\Ind_{\Ll,\Pp}^\Gg\tau^\g
\ =\ \bigoplus_{\g\in\tLl/\Nn_{\tLl}(\tau)}\bigoplus_{i=1}^\ell\pi_i^\g.
\end{equation}

\ref{lem:paramtau.i} Suppose~$w$ normalizes~$\tilde\tau$, in which
case
\[
\Ind_{\tLl,\tPp}^\tGg\tilde\tau\ =\ \tilde\pi_1\oplus\tilde\pi_2, 
\qquad\dim(\tilde\pi_1)>\dim(\tilde\pi_2).
\]
Since the dimensions are distinct, comparing
with~\eqref{eqn:taumackey}, we must have~$\ell=2$ and
\[
\Res_{\Gg}^{\tGg}\tilde\pi_i = \bigoplus_{\g\in\tLl/\Nn_{\tLl}(\tau)}\pi_i^\g;
\]
in particular we
get~$\dim(\pi_1)/\dim(\pi_2)=\dim(\tilde\pi_1)/\dim(\tilde\pi_2)$ so
that the relations of the quadratic generators
in~$\End_\Gg(\Ind_{\Ll,\Pp}^\Gg\tau)$
and~$\End_\tGg(\Ind_{\tLl,\tPp}^\tGg\tilde\tau)$ are the same.

\ref{lem:paramtau.ii} Suppose now that~$w$ does not
normalize~$\tilde\tau$ but does normalize~$\tau$ (so that~$\ell=2$).
Then~$\Ind_{\tLl,\tPp}^\tGg\tilde\tau$ is irreducible so the two pieces
of~$\Ind_{\Ll,\Pp}^\Gg\tau$ are conjugate under~$\tGg^\Ff$ so of the
same dimension; in particular we get~$f_\tau=0$, as required.
\end{proof}

It remains now to compute when the representation~$\tilde\tau$
in~$\Ee(\Ll^\Ff,\tilde s)$ containing~$\tau$ is normalized
by~$w$. Ultimately, the answer depends on the
family of groups in question, but we can already make a preliminary
reduction. We assume from now on that~$w$ does indeed normalize~$\tau$.

We write~$\Ll^\Ff\simeq\GL_n^\Ff\times\Gg_0^\Ff$,
where~$\Gg_0$ is a (possibly trivial) classical group in the same
family as~$\Gg$, and~$\tau=\tau_1\otimes\tau_0$, so that~$\tau_1$ is a
self-dual irreducible cuspidal representation of~$\GL_n^\Ff$. We
write~$\tGg_0$ for the similitude group into which~$\Gg_0$ embeds;
if~$\Gg_0$ is a trivial group (\ie acts on a trivial space)
then~$\tGg_0$ is the multiplicative group.

We also have an isomorphism~$\tLl^\Ff\simeq\GL_n^\Ff\times\tGg_0^\Ff$,
which can be seen as follows. We choose a Witt 
basis~$\ee_N^-,\ldots,\ee_1^-,\ee_1^+,\ldots,\ee_N^+$ for the space on
which~$\Gg^\Ff$ acts, with respect to which~$\Ll^\Ff$ is standard (so
it is the stabilizer of the
subspace~$\la\ee_N^-,\ldots,\ee_{N-n+1}^-\ra$), where we have
changed from our usual notation in the case of a non-split special
orthogonal group, with~$\ee_1^-,\ee_1^+$ a basis for the
anisotropic part of the space. We
write~$\mu:\tGg_0^\Ff\to\FF_q^\times$ for the similitude map, which is
the identity map when~$\Gg_0$ is trivial. Then we get an
isomorphism~$\GL_n^\Ff\times\tGg_0^\Ff\to\tLl$ from the map
\[
(g,h)\ \mapsto\ \diag(g,h,\mu(h)w_n (g^{-1})^\T w_n^{-1}),\qquad
\text{for }g\in\GL_n^\Ff,\ h\in\tGg_0^\Ff,
\]
where~$w_n$ is the antidiagonal element of~$\GL_n^\Ff$ with all
non-zero entries equal to~$1$ (a representative for the longest
element of the Weyl group),~$g^\T$ denotes the transpose matrix, and
the matrix on the right hand side is block diagonal. (Note that,
when~$\Gg_0$ is trivial, the central term~$h$ in the block diagonal
matrix is acting on a trivial space, so is not really present.)
Thus we can write~$\tilde\tau=\tau_1\otimes\tilde\tau_0$,
where~$\tilde\tau_0$ is an irreducible cuspidal representation
of~$\tGg_0^\Ff$ whose restriction to~$\Gg_0^\Ff$ contains~$\tau_0$.

We denote by~$\e$ the sign such that~$\Gg$ preserves an~$\e$-symmetric
form. Then, if either~$\Gg$ is a symplectic group or~$n$ is even, we
can take the representative~$w$ given by
\[
w(\ee_i^\pm)\ =\ \begin{cases}
\ee_i^\pm&\text{ if }1\le i\le N-n, \\
\e\ee_i^{\mp}&\text{ if }N-n<i\le N;
\end{cases}
\]
thus~$w$ is block antidiagonal (for the block sizes corresponding
to~$\Ll$), and the conjugation action of~$w$ 
on~$\tLl$, transported back to~$\GL_n^\Ff\times\tGg_0^\Ff$, is given
by
\begin{equation}\label{eqn:conjw}
(g,h)\ \mapsto\ (\mu(h)(g^{-1})^\T,h).
\end{equation}
Since~$\tau_1$ is self-dual, if its central
character is also trivial then the map~\eqref{eqn:conjw} clearly
intertwines~$\tau_1\otimes\tilde\tau_0$ with itself. The only case
when the central character is non-trivial is when~$\tau_1$ is the
quadratic character, where we see that~\eqref{eqn:conjw}
intertwines~$\tau_1\otimes\tilde\tau_0$ with itself if and only
if~$\tilde\tau_0\simeq\tilde\tau_0\otimes(\tau_1\circ\mu)$. 

This leaves the case of even special orthogonal groups
with~$n=1$, where we need a different representative~$w$. 
Note that, in this case, we do not have that~$\Gg_0$ is
the trivial group, since we have excluded the case~$\Gg=\SO_2^+$. 
We have an identification~$\Gg_0\simeq\SO_{2(N-1)}^\pm$, from
the action on~$\la\ee_{N-1}^-,\ldots,\ee_{N-1}^+\ra$, and we
pick~$c_0\in\O_{2(N-1)}^\pm\setminus\SO_{2(N-1)}^\pm$. Then we can
take~$w$ to be the element given by
\[
\begin{cases} w(\ee_N^\pm)=\ee_N^\mp, \\
w|_{\la\ee_{N-1}^-,\ldots,\ee_{N-1}^+\ra}=c_0. 
\end{cases}
\]
The action of~$w$ on~$\tLl$ is given by
\[
(g,h)\ \mapsto\ (\mu(h)g,c_0hc_0^{-1})
\]
and we see that~$\tilde\tau$ is normalized by~$w$ if and only
if~$\tilde\tau_0\simeq\tilde\tau_0^{c_0}\otimes(\tau_1\circ\mu)$. 

If we set~$c_0=1$ in the case of symplectic groups, we can unify the
discussion above into the following statement.

\begin{lemma}\label{lem:wnormrhotilde}
\begin{enumerate}
\item If~$n>1$ then~$\tilde\tau$ is normalized by~$w$.
\item\label{lem:wnormrhotilde.ii} 
If~$n=1$ then~$\tilde\tau$ is normalized by~$w$ if and only
if~$\tilde\tau_0\simeq\tilde\tau_0^{c_0}\otimes(\tau_1\circ\mu)$.
\end{enumerate}
\end{lemma}

In the following subsections, we analyze precisely when the conditions
in Lemma~\ref{lem:wnormrhotilde}\ref{lem:wnormrhotilde.ii} are
satisfied, in terms of the eigenvalues of the semisimple element~$s$
such that~$\tau\in\Ee(\Ll^\Ff,s)$. 

\subsection{Symplectic groups}

We begin with the case of symplectic groups, so that we are in
case~(ii) of Section~\ref{SS.disconnected}. Thus we have a cuspidal
representation~$\tau$ in~$\Ee(\Ll^\Ff,s)$, for some maximal proper
Levi subgroup~$\Ll$ of~$\Gg$, and a cuspidal representation~$\tilde\tau$
in~$\Ee(\tLl^\Ff,\tilde s)$ whose restriction to~$\Ll$
contains~$\tau$. We denote by~$w$ a representative for the non-trivial element
of~$\Nn_\Gg(\Ll)/\Ll$, which we assume normalizes~$\tau$.

We write~$\Ll^\Ff\simeq\GL_n^\Ff\times\Gg_0^\Ff$,
where~$\Gg_0$ is a (possibly trivial) symplectic group,
and~$\tau=\tau_1\otimes\tau_0$, so that~$\tau_1$ is a self-dual
irreducible cuspidal representation
of~$\GL_n^\Ff$. Then~$\tau_1\in\Ee(\GL_n^\Ff,s_1)$
and~$\tau_0\in\Ee(\Gg_0^\Ff,s_0)$, for some semisimple
elements~$s_0,s_1$ of the respective dual groups. Note that,
if~$\Gg_0$ is the trivial symplectic group then its dual group
is~$\SO_1$ so that~$s_0=1$.

\begin{lemma}\label{lem:wnormrhotildeSp}
The representation~$\tilde\tau$ is normalized by~$w$ unless~$\tau_1$
is the non-trivial quadratic character of~$\GL_1^\Ff$ and~$-1$ is not an
eigenvalue of~$s_0$.
\end{lemma}

We remark that, since~$1$ is always an eigenvalue of~$s_0$ in the case
of symplectic groups, the condition that~$-1$ not be an eigenvalue
of~$s_0$ is equivalent to the condition that~$\tau_0$ extend to the
similitude group~$\tGg_0^\Ff$, by Lemma~\ref{lem:extend}.

\begin{proof}
By Lemma~\ref{lem:wnormrhotilde}, we only need to consider the case
that~$\tau_1$ is the non-trivial quadratic character; in that
case,~$w$ normalizes~$\tilde\tau$ if and only
if~$\tilde\tau_0\simeq\tilde\tau_0\otimes\chi$, where we recall
that~$\tilde\tau_0$ is an irreducible cuspidal representation of the
similitude group~$\tGg_0^\Ff=\GSp_{2(N-1)}^{\Ff}$ containing~$\tau_0$,
and~$\chi=\tau_1\circ\mu$ is the non-trivial quadratic character
of~$\tGg_0^\Ff$. Denote by~$\Zz(\tGg_0^\Ff)$ the centre
of~$\tGg_0^\Ff$; then~$\tilde\tau_0\otimes\chi\simeq\tilde\tau_0$ if
and only if the restriction of~$\tilde\tau_0$ to the index two
subgroup~$\Zz(\tGg_0^\Ff)\Gg_0^\Ff$ is reducible, \ie~$\tilde\tau_0$
restricts reducibly
to~$\Gg_0^\Ff$. Thus~$\tilde\tau^w\simeq\tilde\tau$ if and only
if~$\tau_0$ does \emph{not} extend to~$\tilde\Gg_0^\Ff$, and we are done.
\end{proof}

\subsection{Even special orthogonal groups}

Now we turn to the case of even-dimensional orthogonal groups, so 
that we are in
case~(iii) of Section~\ref{SS.disconnected}. Here, as well as proving
the analogue of Lemma~\ref{lem:wnormrhotildeSp}, we need to consider
the cases, seen in Section~\ref{S.reduction}, 
where the parameter in the affine Hecke algebra is zero, rather than
matching the parameter in the finite Hecke algebra for the connected
reductive quotient. This happens precisely when either we have a
``trivial orthogonal group'' or an irreducible cuspidal representation
of an even-dimensional special orthogonal group which is not
normalized by the full orthogonal group. Thus we must also identify when
this happens, in the language of the previous sections. We begin with
this question, since the answer is also needed for the proof of the
analogue of Lemma~\ref{lem:wnormrhotildeSp}.

\begin{proposition}\label{prop:orthogonalextension}
Let~$\tau\in\Ee(\SO_{2N}^{\pm,\Ff},s)$ be an irreducible cuspidal
representation, and let~$\tilde\tau\in\Ee(\GSO_{2N}^{\pm,\Ff},\tilde
s)$ be an irreducible cuspidal representations, where~$\tilde s$ is a
semisimple element mapping to~$s$. 
\begin{enumerate}
\item\label{prop:orth.ii} 
$\tilde\tau$ extends to a representation of~$\GO_{2N}^{\pm,\Ff}$
if and only if~$1$ is an eigenvalue of~$s$. 
\item\label{prop:orth.i} 
$\tau$ extends to a representation of~$\O_{2N}^{\pm,\Ff}$ if and
only if at least one of~$\pm 1$ is an eigenvalue of~$s$.
\end{enumerate}
\end{proposition}

We remark that in the case that~$s^2=1$ (that is,~$\tau$ is a
\emph{quadratic unipotent representation}),~\ref{prop:orth.ii} is
proved in~\cite[Lemma~8.9]{L77}, while~\ref{prop:orth.i} is proved
in~\cite{Wald} (see the proof of \emph{op.\ cit.}
Proposition~4.3). Our proofs in the general case are similar.

\begin{proof}
In order to prove this, we need to recall a little
about the dual group of~$\GSO_{2N}^\pm$. This dual group is the
\emph{special Clifford group}~$\C^0(\Vv)$, which sits in an exact sequence
\[
1\to\GG_m\to\C^0(\Vv)\to\SO(\Vv)\to 1,
\]
which is exact on points by Hilbert's Theorem~90. It is the connected
component of the full Clifford group~$\C(\Vv)$, which sits in
a similar exact sequence, mapping onto the full orthogonal
group~$\O(\Vv)$. The Clifford group is a subgroup of the group of
invertible elements of the Clifford algebra~$\A=\A(\Vv)$, which
is~$\ZZ/2\ZZ$-graded~$\A=\A^0\oplus \A^1$;
then~$\C(\Vv)=\C^0(\Vv)\sqcup\C^1(\Vv)$, with~$\C^i(\Vv)=\C(\Vv)\cap
\A^i$, and the special Clifford group is a subgroup of index two in the
full Clifford group. 

We will need to know when the semisimple element~$\tilde s\in\C^0(\Vv)$
of the special Clifford group is centralized by some element
of~$\C^1(\Vv)$. If~$1$ is an eigenvalue of~$s$ then it has an
anisotropic eigenvector~$\vv$ with eigenvalue~$1$ and one checks
that~$\vv\in\C^1(\Vv)$ is centralized by~$\tilde s$. If~$1$ is not an
eigenvalue of~$s$, then the elements of~$\A^1$ which commute
with~$\tilde s$ are linear combinations of elements of the form
\[
\vv_1\cdots\vv_r,
\]
with~$\vv_i$ linearly independent eigenvectors of~$s$ with
eigenvalue~$\z_i$, such that~$\prod_{i=1}^r\z_i=1$ and~$r$ is
odd. However, any two such elements anti-commute (since~$r$ is odd)
and any such element squares to~$0$, since the~$\vv_i$ are isotropic
unless~$\z_i=-1$ and we cannot have all~$\z_i=-1$, since their product
is~$1$ and~$r$ is odd. Thus any element of~$a\in \A^1$ which commutes
with~$s$ satisfies~$a^2=0$ so is non-invertible. Thus no element
of~$\C^1(\Vv)$ commutes with~$\tilde s$.

Finally, we pick~$c\in\O_{2N}^{\pm,\Ff}\setminus\SO_{2N}^{\pm,\Ff}$
and we are ready to begin the proof.

\ref{prop:orth.ii}\ The representation~$\tilde\tau$ extends
to~$\GO_{2N}^{\pm,\Ff}$ if and only if it is normalized by
some~$c$. 
Now~$\tau^c\in\Ee(\SO_{2N}^{\pm,\Ff},\tilde c^{-1}\tilde s\tilde c)$,
for some~$\tilde c\in\C^1(\Vv)$.
On the other, hand, these Lusztig series contain only one cuspidal
representation each, so~$\tilde\tau\simeq\tilde\tau^c$ if and only
if~$\tilde s$ is conjugate in~$\C^0(\Vv)$ 
to~$\tilde c^{-1}\tilde s\tilde c$, that
is, if and only if the centralizer in~$\C(\Vv)$ of~$\tilde s$
is~\emph{not} contained in the special Clifford group~$\C^0(\Vv)$. 
However, we have seen above that this happens if and only if~$1$ is an
eigenvalue of~$s$.

\ref{prop:orth.i}\ The proof begins in a similar way. The
representation~$\tau$ extends to~$\O_{2N}^{\pm,\Ff}$ if and only if it
is normalized by~$c$. In this
case,~$\tau^c\in\Ee(\SO_{2N}^{\pm,\Ff},c^{-1}sc)$ and, if this Lusztig
series contains only one cuspidal representation,
then~$\tau\simeq\tau^c$ if and only if~$s$ is conjugate
in~$\SO_{2N}^{\pm,\Ff}$ to~$c^{-1}sc$, that is, if and only if the
centralizer in~$\O_{2N}^{\pm,\Ff}$ of~$s$ is~\emph{not} contained in the special
orthogonal group. On the other hand, we already
know that the Lusztig series contains only one cuspidal representation if and only if
at most one of~$\pm 1$ is an eigenvalue of~$s$, while the centralizer
of~$s$ is contained in~$\SO_{2N}^{\pm,\Ff}$ if and only if
neither~$\pm 1$ is an eigenvalue of~$s$. 

This leaves the case when \emph{both}~$\pm 1$ are eigenvalues
of~$s$. Let~$\tilde\tau$ be an irreducible cuspidal representation
in~$\Ee(\GSO_{2N}^{\pm,\Ff},\tilde s)$ whose restriction
to~$\SO_{2N}^{\pm,\Ff}$ contains~$\tau$. By~\ref{prop:orth.ii}, this
representation is normalized by~$c$.

Now we consider the representation~$\boI\otimes\tau$
of~$\Ll^\Ff=\GL_1^\Ff\times\SO_{2N}^{\pm,\Ff}$, a Levi subgroup
of~$\SO_{2(N+1)}^{\pm,\Ff}$. As at the end of
Section~\ref{SS.disconnected}, we choose a Witt 
basis~$\ee_{N+1}^-,\ldots,\ee_{N+1}^+$ with
respect to which~$\Ll^\Ff$ is standard and denote by~$w$ the element
of~$\SO_{2(N+1)}^{\pm,\Ff}$ given by
\[
\begin{cases} w(\ee_{N+1}^\pm)=\ee_{N+1}^\mp, \\
w|_{\la\ee_{N}^-,\ldots,\ee_{N}^+\ra}=c. 
\end{cases}
\]
Similarly, we have the representation~$\boI\otimes\tilde\tau$
of~$\tLl^\Ff=\GL_1^\Ff\times\GSO_{2N}^{\pm,\Ff}$, which is normalized
by~$w$, by Lemma~\ref{lem:wnormrhotilde}\ref{lem:wnormrhotilde.ii}. 
But then Lemma~\ref{lem:paramtau}\ref{lem:paramtau.i} implies
that~$w$ also normalizes~$\boI\otimes\tau$, whence~$\tau$ is
normalized by~$c$. Thus~$\tau$ extends to~$\O_{2N}^{\pm,\Ff}$, as required.
\end{proof}

Now we return to the notation at the end of
section~\ref{SS.disconnected}. Thus, with~$\Gg=\SO_{2N}^{\pm}$
and~$\tGg=\GSO_{2N}^{\pm}$, we have a cuspidal
representation~$\tau$ in~$\Ee(\Ll^\Ff,s)$, for some maximal proper
Levi subgroup~$\Ll$ of~$\Gg$, and a cuspidal representation~$\tilde\tau$
in~$\Ee(\tLl^\Ff,\tilde s)$ whose restriction to~$\Ll$
contains~$\tau$. We denote by~$w$ a representative for the non-trivial element
of~$\Nn_\Gg(\Ll)/\Ll$, which we assume normalizes~$\tau$.

We write~$\Ll^\Ff\simeq\GL_n^\Ff\times\Gg_0^\Ff$,
where~$\Gg_0$ is a (possibly trivial) special orthogonal group,
and~$\tau=\tau_1\otimes\tau_0$, so that~$\tau_1$ is a self-dual
irreducible cuspidal representation
of~$\GL_n^\Ff$. Then~$\tau_1\in\Ee(\GL_n^\Ff,s_1)$
and~$\tau_0\in\Ee(\Gg_0^\Ff,s_0)$, for some semisimple
elements~$s_0,s_1$ of the respective dual groups. Note that,
if~$\Gg_0$ is the trivial special orthogonal group then its dual group
is the trivial group so that~$s_0$ has no eigenvalues.

\begin{corollary}\label{cor:wnormrhotildeSO}
The representation~$\tilde\tau$ is normalized by~$w$ in all but the
following cases: 
\begin{enumerate}
\item $\tau_1$ is the trivial character of~$\GL_1^\Ff$ and~$1$ is not
an eigenvalue of~$s_0$;
\item $\tau_1$ is the quadratic character of~$\GL_1^\Ff$ and~$-1$ is
not an eigenvalue of~$s_0$.
\end{enumerate}
\end{corollary}

We remark that, among the exceptional cases, we cannot have
that~$\Gg_0$ is the trivial group, since otherwise we would
have~$\Gg=\SO_2^+$, which we have excluded.

\begin{proof}
By Lemma~\ref{lem:wnormrhotilde}, we need only consider the
case~$n=1$, so that~$\tau_1$ is a trivial or quadratic character; 
then~$\tilde\tau$ is normalized by~$w$ if and only
if~$\tilde\tau_0\simeq\tilde\tau_0^{c_0}\otimes(\tau_1\circ\mu)$,
where we recall that~$\tilde\tau_0$ is an irreducible cuspidal
representation of the similitude
group~$\tGg_0^\Ff=\GSO_{2(N-1)}^{\pm,\Ff}$ containing~$\tau_0$,
and~$c_0\in\O_{2(N-1)}^{\pm,\Ff}\setminus\SO_{2(N-1)}^{\pm,\Ff}$.

Suppose first that~$\tau_1$ is trivial. Then~$\tilde\tau$ is
normalized by~$w$ if and only if~$\tilde\tau_0$ is normalized
by~$c_0$, which happens if and only if~$\tilde\tau_0$ extends to the
full similitude group; by Proposition~\ref{prop:orthogonalextension},
this happens if and only if~$1$ is an eigenvalue of~$s_0$, and we are
done. 

Now suppose that~$\tau_1$ is the non-trivial quadratic character and
put~$\chi_1=\tau_1\circ\mu$. We also denote by~$c_1$ an element
of~$\GSO_{2(N-1)}^{\pm,\Ff}$ which is not
in~$\Zz(\GSO_{2(N-1)}^{\pm,\Ff})\SO_{2(N-1)}^{\pm,\Ff}$; 
thus~$\tilde\tau_0$ is an extension of~$\tau_0$ if and only
if~$\tau_0$ is normalized by~$c_1$, if and only 
if~$\tilde\tau_0\simeq\tilde\tau_0\otimes\chi_1$.

If both~$\pm 1$ are eigenvalues of~$s_0$, 
then~$\tilde\tau_0\simeq\tilde\tau\otimes\chi_1$, 
while~$\tilde\tau_0$ extends to~$\GO_{2(N-1)}^{\pm,\Ff}$, by 
Proposition~\ref{prop:orthogonalextension}\ref{prop:orth.ii}. 
Thus~$\tilde\tau_0^{c_0}\otimes\chi_1\simeq\tilde\tau_0\otimes\chi_1
\simeq\tilde\tau_0$.

If neither~$\pm 1$ is an eigenvalue of~$s_0$, 
then~$\tilde\tau_0^{c_0}\otimes\chi_1$ is an extension 
of~$\tau_0^{c_0}$, which is not equivalent to~$\tau_0$ by 
Proposition~\ref{prop:orthogonalextension}\ref{prop:orth.i}. 
Thus~$\tilde\tau_0^{c_0}\otimes\chi_1\not\simeq\tilde\tau_0$.

Finally, if exactly one of~$\pm 1$ is an eigenvalue of~$s_0$,
then~$\tau_0^{c_0}\simeq\tau_0$, by
Proposition~\ref{prop:orthogonalextension}\ref{prop:orth.i}. 
Then~$\tilde\tau_0^{c_0}$ contains~$\tau_0$ on restriction, 
so is equivalent either
to~$\tilde\tau_0$ or to~$\tilde\tau_0\otimes\chi_1$, since it agrees
with~$\tilde\tau_0$ on the index two
subgroup~$\Zz(\GSO_{2(N-1)}^{\pm,\Ff})\SO_{2(N-1)}^{\pm,\Ff}$
of~$\GSO_{2(N-1)}^{\pm,\Ff}$. By
Proposition~\ref{prop:orthogonalextension}\ref{prop:orth.ii}, the
former happens if and only if~$1$ is an eigenvalue of~$s_0$, in which
case~$\tilde\tau_0^{c_0}\otimes\chi_1\simeq\tilde\tau_0\otimes\chi_1\not\simeq\tilde\tau_0$. Thus
the latter happens if and only if~$-1$ is an eigenvalue~$s_0$, in
which case~$\tilde\tau_0^{c_0}\otimes\chi_1\simeq\tilde\tau_0$. 
\end{proof}

\subsection{Summary}\label{SS.summary}

We summarize the results of all these calculations, including looking up
parameters in Lusztig's tables in~\cite{LCBMS}, in the following
table. We are given an irreducible cuspidal representation~$\tau$
in~$\Ee(\Ll^\Ff,s)$, for some maximal proper Levi subgroup~$\Ll$
of~$\Gg$, which is normalized by~$\Nn_\Gg(\Ll)$. We
write~$\Ll^\Ff\simeq\GL_n^\Ff\times\Gg_0^\Ff$, where~$\Gg_0$ is a
(possibly trivial) classical group of the same type as~$\Gg$,
and~$\tau=\tau_1\otimes\tau_0$, so that~$\tau_1$ is a self-dual
irreducible cuspidal representation
of~$\GL_n^\Ff$. Then~$\tau_1\in\Ee(\GL_n^\Ff,s_1)$
and~$\tau_0\in\Ee(\Gg_0^\Ff,s_0)$, for some semisimple
elements~$s_0,s_1$ of the respective dual groups. 

We write
\[
P_{s_0}(X)=\prod_P P(X)^{a_P}(X-1)^{a_+}(X+1)^{a_-}
\]
for the characteristic polynomial of~$s_0$, where the product is over
all irreducible self-dual monic polynomials over~$\FF_q$ of even
degree, and the integers~$a_P,a_{\pm}$ are related to
integers~$m_P,m_\pm$ as in the description in~\ref{eqn:notitle}. We
also write~$Q$ for the characteristic
polynomial of~$s_1\in\GL_n^{*,\Ff}$; thus either~$Q(X)=(X\pm 1)$ or~$Q$
is an irreducible self-dual monic polynomial of even
degree~$n=n_Q$. In the table, the cases (i)--(iii) refer to the
different possible classical groups, as in Section~\ref{SS.disconnected}.

\begin{center}
\begin{tabular}{|c|c|c|c|} 
\hline
degree~$n$ & polynomial~$Q$ & case & $f_\tau$ \\ \hline \hline
$1$ & $X-1$ & {\begin{tabular}{c} (i),(ii) \\ (iii) \end{tabular}}
& {\begin{tabular}{c} $2m_++1$ \\ $2m_+$ \end{tabular}} \\ \hline \hline
$1$ & $X+1$ & {\begin{tabular}{c} (i) \\ 
  (ii),(iii) \end{tabular}} & {\begin{tabular}{c} $2m_-+1$
  \\ $2m_-$ \end{tabular}}\\ \hline\hline
$n_Q$ even & \phantom{$\dfrac{2^2}{y_y}$}$Q$\phantom{$\dfrac{2^2}{y_y}$} & (i),(ii),(iii)& $(2m_Q+1)\dfrac{n_Q}2$ \\
\hline 
\end{tabular}
\end{center}

\subsection{Unitary groups}\label{SS.unitary}

Finally, we consider the case of unitary groups. We could have
included this in the cases of
Sections~\ref{SS.connected}--\ref{SS.summary} above but it would
have further complicated the
notation. Instead, we indicate here the differences with the previous
cases and summarize the final results.

Let~$q=q_\so^2$ be an even power of an odd prime~$p$, 
take~$\Gg=\GL_n$ over the finite field~$\FF_{q_\so}$, and 
let~$\Ff$ be the twisted Frobenius map, so
that~$\Gg^\Ff$ is a unitary group (which we can think of as a subgroup
of~$\GL_n(\FF_{q})$). Then~$\Gg^*=\GL_n$ act naturally on
an~$n$-dimensional vector space~$\Vv$ with
an~$\FF_q/\FF_{q_\so}$-hermitian form. For~$s\in\Gg^{*,\Ff}$
semisimple, we denote 
by~$P_s(X)\in\FF_{q}[X]$ its characteristic polynomial as an
automorphism of~$\Vv$. 

From~\cite[\S9]{L77}, the equivalence classes of irreducible cuspidal
representations of~$\Gg^\Ff$ are in bijection with the set of
conjugacy classes of semisimple elements~$s$ in~$\Gg^{*,\Ff}$ whose
characteristic polynomial is of the form
\[
P_s(X)=\prod_P P(X)^{a_P},
\]
where the product runs over all irreducible~$\FF_q/\FF_{q_\so}$-self-dual monic
polynomials in~$\FF_{q}[X]$ (see Section~\ref{SS.polynomials}),
and~$a_P=\tfrac 12(m_P^2+m_P^{})$, for some integer~$m_P^{}\ge 0$. 

Now suppose~$\Ll$ is a maximal proper~$\Ff$-stable Levi subgroup
of~$\Gg$ contained in an~$\Ff$-stable parabolic subgroup~$\Pp$. We
write~$\Ll^\Ff\simeq\GL_m(\FF_q)\times\Gg_0^\Ff$, with~$\Gg_0^\Ff$ again
a unitary group. 
Let~$\tau$ be an irreducible cuspidal representation of~$\Ll^\Ff$ with
the property that any representative~$w$ for the non-trivial element
of~$\Nn_{\Gg}(\Ll)/\Ll$ normalizes~$\tau$. Thus we may
decompose~$\tau=\tau_1\otimes\tau_0$, with~$\tau_1$ a
(conjugate)-self-dual irreducible cuspidal representation
of~$\GL_m(\FF_q)$ and~$\tau_0$ an irreducible cuspidal representation
of~$\Gg_0^\Ff$. 

In this situation, the induced representation~$\Ind_{\Ll,\Pp}^\Gg\tau$
decomposes again as~$\pi_1\oplus\pi_2$,
with~$\dim(\pi_1)>\dim(\pi_2)$,
and~$\End_{\Gg^\Ff}(\Ind_{\Ll,\Pp}^{\Gg}\tau)$ is a two-dimensional
algebra with a quadratic generator~$T$ satisfying a relation of the form
\[
(T+1)(T-q^{f_\tau})\ =\ 0,\qquad q^{f_\tau}=\frac{\dim(\pi_1)}{\dim(\pi_2)}>1.
\]
As in the connected case above, the parameter may be computed via the
Jordan decomposition of characters and Lusztig's tables, as
follows. For~$s_0$ a semisimple element of~$\Gg_0^{*,\Ff}$ such
that~$\tau_0\in\Ee(\Gg_0^\Ff,s_0)$, we write its characteristic
polynomial
\[
P_{s_0}(X)=\prod_P P(X)^{a_P},
\]
for integers~$a_P=\tfrac 12(m_P^2+m_P^{})$ as above. We also write~$Q$
for the irreducible characteristic polynomial of an
element~$s_1\in\GL_m^{*,\Ff}$ such that~$\tau_1\in\Ee(\GL_m^\Ff,s_1)$;
thus~$Q$ is an irreducible~$\FF_q/\FF_{q_\so}$-self-dual monic
polynomial, 
of some odd degree~$n=n_Q$. Then we get
\[
f_\tau = (2m_Q+1)\frac{n_Q}2.
\]

\section{Synthesis}
\label{S.synthesis}

In this section, we put together the previous results to verify the
equality~\eqref{eqn:expected}, for~$\pi$ a depth zero irreducible cuspidal
representation of the classical group~$\G$. Recall that~$N_\hG$ is the
dimension of the vector space on which the complex dual group~$\hG$
acts naturally. Strictly speaking, here we only prove the inequality
\[
\sum_{\rho\in\Aa^\s(\F)}\left\lfloor\(s_\pi(\rho)\)^2\right\rfloor n_\rho\ \ge\ N_{\hG},
\]
by checking that the sum over depth zero self-dual irreducible cuspidal representations
already gives us~$N_{\hG}$, that is:
\begin{equation}\label{eqn:expected0bis}
\sum_{\rho\in\Aa^\s_{[0]}(\F)}\left\lfloor\(s_\pi(\rho)\)^2\right\rfloor n_\rho\ =\ N_{\hG},
\end{equation} 
In many cases, the opposite inequality was already proved
by M\oe glin in~\cite{Mo2}; alternatively, the techniques used here,
together with the results in~\cite{MiS}, easily show that, for~$\rho$
a positive depth self-dual irreducible cuspidal representation we
have~$s_\pi(\rho)\in\left\{0,\pm\frac12\right\}$, so that these do not
contribute to the sum. (See~\cite{BHS}, where this is carried out in a
more general situation, for details.)

Thus we return to the notation of
Sections~\ref{S.notation}--\ref{S.reduction}: we
have~$\pi=\cind_{\J_\pi}^\G\l_\pi$ an irreducible cuspidal depth zero representation of a
classical group~$\G$, with~$\J_\pi$ the normalizer of a standard
maximal parahoric subgroup~$\J_{N_1,N_2}^\so$,
and~$\l_\pi|_{\J_{N_1,N_2}^\so}$ contains an irreducible
representation~$\l_\pi^\so$ inflated from an irreducible cuspidal
representation~$\tau_\pi\simeq\tau_\pi^{(1)}\otimes\tau_\pi^{(2)}$ of the
reductive quotient~$\Gg_{N_1,N_2}^\so \simeq
\Gg_{N_1}^{(1)}\times\Gg_{N_2}^{(2)}$. For~$i=1,2$, there is a unique
conjugacy class~$(s_i)$ in~$\Gg_{N_i}^{(i),*}$ such that~$\tau_\pi^{(i)}$ is
in the Lusztig series~$\Ee(\Gg_{N_i}^{(i)},s_i)$, and we denote the
characteristic polynomial of~$s_i$ by
\[
\prod_{P}P(X)^{a_P^{(i)}},
\]
where the product runs over irreducible~$k_\F/k_\so$-self-dual monic
polynomials in~$k_\F[X]$, and the powers~$a_P^{(i)}$
satisfy the conditions of~\ref{eqn:notitle}; in particular, there
are integers~$m_p^{(i)}\ge 0$ such that:
\begin{itemize}
\item if~$k_\F\ne k_\so$ or~$P(X)\ne (X\pm 1)$ then~$a_P^{(i)}=\frac 12 m_P^{(i)}(m_P^{(i)}+1)$;
\item if~$k_\F=k_\so$ and~$P(X)=(X\pm 1)$ then we
 write~$m_+^{(i)}=m_{(X-1)}^{(i)}$ and~$m_-^{(i)}=m_{(X+1)}^{(i)}$, 
 to match the notation of Section~\ref{S.computation},
 and these satisfy the conditions in~\ref{eqn:notitle}.
\end{itemize}

\begin{remark}
It may be that~$N_i=N_i^\an=0$, for~$i=1$ or~$2$; in this case the
group~$\Gg_{N_i}^{(i)}$ is trivial, but we must interpret it as the
``right'' trivial group. That is, if~$\G$ is symplectic then the group
is a trivial symplectic group; if~$\G$ is special orthogonal it is a trivial
special orthogonal group; if~$\G$ is unramified unitary it is a
trivial unitary group; and if~$\G$ is ramified unitary then it is a
trivial symplectic group if~$\e=(-1)^i$, and a trivial special
orthogonal group otherwise. In particular, if the group is trivial
symplectic then the characteristic polynomial of~$s_i$ is~$X-1$;
in the other cases, the characteristic polynomial of~$s_i$ is the
constant polynomial~$1$.
\end{remark}

Now, for~$\rho$ a self-dual irreducible cuspidal depth zero representation of
some~$\GL_{n}(\F)$, we have a unique self-dual irreducible cuspidal
representation~$\tau_\rho$ of~$\GL_n(k_\F)$ such that~$\rho$ contains
the representation~$\l_\rho$ of~$\GL_n(\o_\F)$ obtained from~$\tau_\rho$
by inflation. Then~$\tau_\rho$ is in the Lusztig series associated to
some conjugacy class in~$\GL_n(k_\F)$ with irreducible self-dual
characteristic polynomial~$Q=Q_\rho$ of degree~$n$.

We suppose first that~$k_\F\ne k_\so$ or~$Q(X)\ne(X\pm 1)$; thus
either~$n>1$ or~$\G$ is an unramified unitary group, and the
parameters~$q^{f_i}$ of the Hecke algebra are always computed from the
Hecke algebra in the finite group. Then the formulae in
Sections~\ref{SS.summary}--\ref{SS.unitary} give
\[
f_i = f_{\tau_\rho\otimes\tau_\pi^{(i)}} = (2m_Q^{(i)}+1)\frac{n}{2}
\]
and, from~\eqref{eqn:redpoints} we get reducibility points
\[
\left\{\pm s_\pi(\rho),\pm s_\pi(\rho')\right\} = \left\{\pm \frac{(m_Q^{(1)}+m_Q^{(2)}+1)}{2},\pm\frac{(m_Q^{(1)}-m_Q^{(2)})}{2}\right\}.
\]
Since one of these is an integer and the other a half-integer, we get
\[
\lfloor s_\pi(\rho)^2\rfloor + \lfloor s_\pi(\rho')^2\rfloor
=
\(\frac{(m_Q^{(1)}+m_Q^{(2)}+1)}{2}\)^2 + \(\frac{(m_Q^{(1)}-m_Q^{(2)})}{2}\)^2 -
\frac 14 = a_Q^{(1)} + a_Q^{(2)}.
\]
Thus we are already done in the case of unramified unitary groups:
summing, we get
\[
\sum_{\rho\in\Aa_{[0]}^\s(\F)}\lfloor s_\pi(\rho)^2\rfloor n_\rho =
\sum_{P}(a_P^{(1)}+a_P^{(2)})\deg(P) = (2N_1+N_1^\an)+(2N_2+N_2^\an) = 2N+N^\an,
\]
as required.

For the cases~$k_\F=k_\so$ and~$Q(X)=(X\pm 1)$, we will split
according to the type of group~$\G$, since the values for the
parameters do not admit such a uniform description. 

\subsection{Symplectic groups}
We suppose first that~$Q(X)=X-1$, so that~$\rho,\rho'$ are the trivial
character and the unramified character of order~$2$, and
write~$m_Q^{(i)}=m_+^{(i)}$. Since both~$\Gg_{N_i}^{(i)}$ are
symplectic groups, we get
\[
f_i = (2m_+^{(i)}+1),
\]
with reducibility points
\[
\left\{\pm s_\pi(\rho),\pm s_\pi(\rho')\right\} = \left\{\pm
 (m_+^{(1)}+m_+^{(2)}+1),\pm (m_+^{(1)}-m_+^{(2)})\right\}.
\]
Thus
\[
\lfloor s_\pi(\rho)^2\rfloor + \lfloor s_\pi(\rho')^2\rfloor
=
(m_+^{(1)}+m_+^{(2)}+1)^2 + (m_+^{(1)}-m_+^{(2)})^2 = a_+^{(1)} + a_+^{(2)} - 1.
\]

Now suppose that~$Q(X)=X+1$, so that~$\rho,\rho'$ are (tamely) ramified
characters of order~$2$, and write~$m_Q^{(i)}=m_-^{(i)}$. Then we get
\[
f_i = 2m_-^{(i)},
\]
with reducibility points
\[
\left\{\pm s_\pi(\rho),\pm s_\pi(\rho')\right\} = \left\{\pm
 (m_-^{(1)}\pm m_-^{(2)})\right\}.
\]
Thus
\[
\lfloor s_\pi(\rho)^2\rfloor + \lfloor s_\pi(\rho')^2\rfloor
=
(m_-^{(1)}+m_-^{(2)})^2 + (m_-^{(1)}-m_-^{(2)})^2 = a_-^{(1)} + a_-^{(2)}.
\]

Finally, summing we get
\[
\sum_{\rho\in\Aa^\s_{[0]}(\F)}\lfloor s_\pi(\rho)^2\rfloor n_\rho =
\sum_{P}(a_P^{(1)}+a_P^{(2)})\deg(P) - 1 = (2N_1+1)+(2N_2+1) -1 = 2N+1.
\]

\subsection{Ramified unitary groups}
In this case, the groups~$\Gg_{N_i}^{(i)}$ are one symplectic and one
orthogonal; for ease of exposition, we will assume that~$\Gg_{N_1}^{(1)}$
is symplectic (otherwise exchange~$1$ and~$2$).

We begin again with the case~$Q(X)=X-1$ and write~$m_+^{(i)}$ in place
of~$m_Q^{(i)}$. Thus we get
\[
f_1 = 2m_+^{(1)}+1,\ f_2=
\begin{cases} 2m_+^{(2)}+1,&\hbox{ if }N_2^\an=1, \\
2m_+^{(2)},&\hbox{ otherwise.}\end{cases}
\]
Thus we get reducibility points
\[
\left\{\pm s_\pi(\rho),\pm s_\pi(\rho')\right\} = 
\begin{cases}
\left\{\pm (m_+^{(1)}+ m_+^{(2)}+1), \pm (m_+^{(1)}-m_+^{(2)})\right\}, &\hbox{ if
}N_2^\an=1,\\
\left\{\pm \(m_+^{(1)}+m_+^{(2)}+\tfrac 12\), \pm \(m_+^{(1)}-m_+^{(2)}+\tfrac
 12\)\right\}, &\hbox{ otherwise.}
\end{cases}
\]
Thus, if~$N_2^\an=1$, we get
\[
\lfloor s_\pi(\rho)^2\rfloor + \lfloor s_\pi(\rho')^2\rfloor
=
(m_+^{(1)}+m_+^{(2)}+1)^2 + (m_+^{(1)}-m_+^{(2)})^2 = a_+^{(1)} + a_+^{(2)};
\]
and otherwise, since both reducibility points are half-integers, we get
\[
\lfloor s_\pi(\rho)^2\rfloor + \lfloor s_\pi(\rho')^2\rfloor
=
\(m_+^{(1)}+m_+^{(2)}+\tfrac 12\)^2 + \(m_+^{(1)}-m_+^{(2)}+\tfrac 12\)^2 -\tfrac 12 = a_+^{(1)} + a_+^{(2)}-1.
\]

The case~$Q(X)=X+1$ is similar, the main difference being
that~$f_1=2m_-^{(1)}$. Then we get reducibility points
\[
\left\{\pm s_\pi(\rho),\pm s_\pi(\rho')\right\} = 
\begin{cases}
\left\{\pm (m_-^{(1)}+ m_-^{(2)}+\tfrac 12), \pm (m_-^{(1)}-m_-^{(2)}-\tfrac 12)\right\}, &\hbox{ if
}N_2^\an=1,\\
\left\{\pm \(m_-^{(1)}\pm m_-^{(2)}\)\right\}, &\hbox{ otherwise.}
\end{cases}
\]
Now in both cases we get
\[
\lfloor s_\pi(\rho)^2\rfloor + \lfloor s_\pi(\rho')^2\rfloor
= a_-^{(1)} + a_-^{(2)}.
\]
Noting that we have
\[
\sum_{P} a_P^{(1)}\deg(P) = 2N_1+1,\quad
\sum_{P} a_P^{(2)}\deg(P) = \begin{cases}
2N_2,&\hbox{ if }N_2^\an=1,\\
2N_2+N_2^\an,&\hbox{ otherwise,}
\end{cases}
\]
we once again see that, summing over all depth zero self-dual irreducible cuspidal
representations of all~$\GL_n(\F)$, equation~\ref{eqn:expected0bis} is satisfied.

\subsection{Special orthogonal groups}
The case of special orthogonal groups is exactly analogous and we do not give
the details. One can check the equality in~\eqref{eqn:expected0bis} 
by working through the
cases according to the parities of~$N_1^\an,N_2^\an$. For example, if
both are odd, then with~$Q(X)=X\pm1$ we get
\[
\lfloor s_\pi(\rho)^2\rfloor + \lfloor s_\pi(\rho')^2\rfloor = a_\pm^{(1)} + a_\pm^{(2)}+1.
\]
The additions of the extra~$1$ here exactly compensate for the fact
that the dual groups of~$\Gg_{N_i}^{(i)}$ have dimension~$2N_i=2N_i+N_i^\an-1$.

\subsection{Summary}
In all cases, we have now checked that the
equality~\eqref{eqn:expected0bis} holds. We have also seen
that~$s_\pi(\rho)\in\frac 12\ZZ$ in all cases.

\section{$\L$-packets and Examples}
\label{S.examples}

In this final section, we examine the implications of the results here
for the computation of~$\L$-packets and give some examples. Firstly, we 
recall some facts about the (expected) sizes of discrete series~$\L$-packets
containing an irreducible cuspidal representation, and the (expected) number 
of cuspidal representations in them. 

Let~$\vphi:\Ww_\F\times\SL_2(\CC)\to\hG\rtimes\Ww_\F$ be a Langlands parameter for~$\G$ 
whose~$\L$-packet~$\Pi_\vphi$ contains an irreducible cuspidal representation~$\pi$ of~$\G$. 
Then, as recalled in the introduction, we should have
\[
\vphi\ =\ \bigoplus_{(\rho,m)\in\Jord(\pi)} \vphi_\rho\otimes\st_m,
\]
where~$\vphi_\rho$ is the (irreducible) representation of the Weil
group~$\Ww_\F$ corresponding to~$\rho$ via the Langlands
correspondence for general linear groups, and~$\st_m$ is
the~$m$-dimensional irreducible representation of~$\SL_2(\CC)$. 
Putting~$\ell(\pi)=\#\Jord(\pi)$, the number of representations one expects in the 
packet~$\Pi_\vphi$ is~$2^{\ell(\pi)-1}$, since this is the number of characters 
of the component group of~$\Cent_{\hG}(\Im(\vphi))$ trivial on the centre 
of~$\hG$. We also set
\[
\E(\pi)\ =\ \left\{\rho\in\Aa^\s(\F)\mid s_\pi(\rho)\in\NN\right\}
\]
and~$e(\pi)=\#\E(\pi)$; this is the number of~$\rho\in\Aa^\s(\F)$ such 
that~$(\rho,m)\in\Jord(\pi)$ for some \emph{odd} integer~$m$. Finally, put
\[
e_0(\pi)\ =\ \begin{cases} 1&\text{ if there is }\rho\in\Aa^\s(\F)
\text{ such that } s_\pi(\rho)\in 1+2\NN,\\
0&\text{ otherwise.}
\end{cases}
\]
Then the number of irreducible cuspidal representations one expects in the~$\L$-packet~$\Pi_\vphi$
is~$2^{e(\pi)-e_0(\pi)}$, since this is the number of characters
of the component group of~$\Cent_{\hG}(\Im(\vphi))$ which are trivial on 
the centre of~$\hG$ and \emph{alternating}, in the sense of, for 
example,~\cite[Section~8]{Mo3}.

\begin{remark}
The reference to the work of M\oe glin in~\cite[Section~8]{Mo3} is in fact only for 
unitary groups, though it also holds when~$G$ is a quasi-split unitary, orthogonal, 
symplectic or~$\mathrm{GSpin}$ group (see~\cite{Mo4}) and it seems reasonable to expect 
it to hold in more generality. Although our results here on the representations with 
given inertial reducibility set do not require it, for the purposes of 
discussion, from now on we make the following assumption:
\begin{itemize}
\item[(A)] The description of the number of irreducible cuspidal representation in an~$\L$-packet above is valid for the group~$\G$.
\end{itemize}
However, extra care must be taken when~$\G$ is an even-dimensional special orthogonal 
group -- see Example~\ref{ex:ortho}.
\end{remark}

\medskip

Now we describe how our results allow us to find all irreducible cuspidal representations with 
the same inertial reducibility set, hence all irreducible cuspidal representations in a union 
of one, two or four~$\L$-packets (assuming (A)). We suppose we are
in the general situation of the previous section, with~$\pi$ an
irreducible cuspidal depth zero representation of~$\G$. Recall that
\[
\IRed(\pi)\ =\ 
\left\{\!\left\{([\rho],m)\mid \rho\in\Aa^\s(\F),
~m\in\NN\text{ with }2s_\pi(\rho)=m+1\right\}\!\right\},
\]
where~$[\rho]$ denotes the inertial equivalence class of~$\rho$.

Our representation~$\pi$ is induced from a
representation containing the inflation of an irreducible cuspidal
representation~$\tau_\pi\simeq\tau_\pi^{(1)}\otimes\tau_\pi^{(2)}$ of the
reductive quotient~$\Gg_{N_1}^{(1)}\times\Gg_{N_2}^{(2)}$ of a maximal
parahoric subgroup. For~$i=1,2$ and~$P$ an irreducible~$k_\F/k_\so$-self-dual
monic polynomial in~$k_\F[X]$, we denote by~$m_P^{(i)}$ the
associated non-negative integer as in~\ref{eqn:notitle}.

The formulae obtained above show that, for
each irreducible~$k_\F/k_\so$-self-dual monic polynomial~$P$ in~$k_\F[X]$,
the pair of integers~$\{m_P^{(1)},m_P^{(2)}\}$ can be recovered from the
reducibility points~$\{s_\pi(\rho),s_\pi(\rho')\}$, for~$\rho=\rho_P$ a
representation in~$\Aa^\s_{[0]}(\F)$ with associated characteristic
polynomial~$P$, and~$\rho'$ its self-dual unramified twist. Indeed,
one gets the following, where we write~$\lvert \,\cdot\,\rvert_\infty$
for the usual (archimedean) absolute value on~$\RR$:
\begin{itemize}
\item if~$P(X)\ne X\pm 1$ or~$k_\F\ne k_\so$ then
\[
\left\{m_P^{(1)},m_P^{(2)}\right\}\ =\ 
\left\{\left\lfloor \lvert s_\pi(\rho)\pm s_\pi(\rho') \rvert_\infty
 \right\rfloor \right\};
\]
\item if~$P(X)= X\pm 1$ and~$k_\F=k_\so$, so that~$\rho$ is a character of~$\GL_1(\F)$
of order at most~$2$, then
\[
\left\{m_P^{(1)},m_P^{(2)}\right\}\ =\ 
\left\{\left\lfloor \frac{\lvert s_\pi(\rho)\pm s_\pi(\rho') \rvert_\infty}2
 \right\rfloor \right\}.
\]
\end{itemize}
Thus one obtains the same inertial reducibility set as for~$\pi$
only for irreducible cuspidal representations~$\pi'$
with~$\{m_P^{(1)}(\pi),m_P^{(2)}(\pi)\}=\{m_P^{(1)}(\pi'),m_P^{(2)}(\pi')\}$,
for every irreducible~$k_\F/k_\so$-self-dual monic polynomial~$P$
in~$k_\F[X]$. Hence, in order to obtain other representations with the
same inertial reducibility set, it is enough to exchange (some of) the
integers~$m_P^{(1)}(\pi),m_P^{(2)}(\pi)$. Note, however, that it is
not always possible to do this, for parity reasons. We set
\[
\Q(\pi)\ =\ \left\{\text{irreducible self-dual monic }P\in k_\F[X]\mid
m_P^{(1)}(\pi)\ne m_P^{(2)}(\pi)\right\}
\]
and put~$q(\pi):=\#\Q(\pi)$. 

\begin{remark}\label{rmk:Q=E}
If~$P\in k_\F[X]$ is an irreducible~$k_\F/k_\so$-self-dual monic polynomial, 
with~$P(X)\ne X\pm 1$ or~$k_\F\ne k_\so$, and~$\rho_P,\rho'_P$ are the 
corresponding self-dual irreducible cuspidal representations of some~$\GL_n(\F)$, then 
one of~$s_\pi(\rho_P),s_\pi(\rho'_P)$ is integral and the other
non-integral. In particular, we see that (exactly) one of~$\rho_P,\rho'_P$ is 
in~$\E(\pi)$ if and only if~$P\in\Q(\pi)$.

A similar analysis can be done for~$P(X)=X\pm 1$ and~$k_\F= k_\so$, but it depends 
on the type of group. We summarize the results in the separate cases below.
\end{remark}

We will parametrize the irreducible cuspidals~$\pi'$ with~$\IRed(\pi')=\IRed(\pi)$ by 
maps~$\bee:\Q(\pi)\to\{1,2\}$, noting that not all such maps are permissible, 
and that each map may give rise to more than one representation.
We split again according to the type of group.

\subsection{Symplectic groups}
We begin with the case that~$\G=\Sp_{2N}(\F)$ is a symplectic group,
in which case there are no restrictions on the maps~$\bee:\Q(\pi)\to\{1,2\}$. 
We put
\[
\d(\pi)=\#\left\{i\mid m_{X+1}^{(i)}(\pi)\ne 0\right\}.
\]
Now suppose we have~$\bee:\Q(\pi)\to\{1,2\}$ and, for
irreducible self-dual monic~$P\not\in\Q(\pi)$, we
set~$\bee(P)=1$. Then we can find, for~$i=1,2$, a semisimple
element~$s_{\bee}^{(i)}$ in a suitable odd special orthogonal
group~$\SO_{2N'_i+1}(k_\F)$ with characteristic polynomial
\[
\prod_P P(X)^{a_P^{(i)}(\bee)},
\]
where the product is taken over all irreducible self-dual monic
polynomials in~$k_\F[X]$ and the integers~$a_P^{(i)}(\bee)$ are
related to integers~$m_P^{(i)}(\bee)$ as in~\ref{eqn:notitle}, with
\[
m_P^{(i)}(\bee)\ =\ m_P^{(i\cdot\bee(P))},
\]
where the index is understood modulo~$3$. Correspondingly, we have
irreducible cuspidal
representations~$\tau_\bee=\tau_\bee^{(1)}\otimes\tau_\bee^{(2)}$
of~$\Gg_{N_1',N_2'}=\Sp_{2N'_1}(k_\F)\times\Sp_{2N'_2}(k_\F)$; note that, for
each~$\bee$, the number of such
representations is~$2^{\d(\pi)}$. Inflating each~$\tau_\bee$ to the
maximal parahoric subgroup~$\J_{N_1',N_2'}$ and inducing to~$\G$, we
get an irreducible cuspidal representation. Thus we
get~$2^{q(\pi)+\d(\pi)}$ irreducible cuspidal representations of~$\G$.

The analysis of Remark~\ref{rmk:Q=E}, along with that for the cases~$P(X)=X\pm
1$, shows that
\[
e(\pi)\ =\ \begin{cases} 
q(\pi)+\d(\pi)+1 &\text{ if }\d(\pi)\le 1; \\
q(\pi)+\d(\pi) &\text{ if }\d(\pi)=2.
\end{cases}
\]
On the other hand, we always have~$e_0(\pi)=1$, since either~$s_\pi(\boI)$ 
or~$s_\pi(\om_0)$ is an odd integer, where~$\boI$ is the trivial character 
of~$\GL_1(\F)$ and~$\om_0$ is the unramified character of order two.
Hence we have constructed the irreducible cuspidal
representations in a union of two~$\L$-packets if~$\d(\pi)=2$, or in a
single~$\L$-packet otherwise.

Thus, in some cases we are able to identify all the representations
in a single~$\L$-packet, but in others we cannot distinguish between
the representations in two~$\L$-packets without further work.
We give some examples to illustrate these phenomena.
In the following, we write~$\om_1,\om_2=\om_0\om_1$ for the
(tamely) ramified characters of~$\GL_1(\F)$ of order two.

\begin{example}\label{ex:yes}
We begin with an example where
we are able to recover all the cuspidal representations in a
single~$\L$-packet. We take~$\G=\Sp_6(\F)$ and begin with the parahoric
subgroup~$\J_{2,1}$, which has reductive
quotient~$\Gg_{2,1}\simeq\Sp_4(k_\F)\times\SL_2(k_\F)$.

We take the representation~$\theta_{10}$ of~$\Sp_4(k_\F)$, that is the
unique cuspidal representation in the Lusztig
series~$\Ee(\Sp_4(k_\F),1)$ (so that the associated characteristic
polynomial is~$(X-1)^5$). We also take an irreducible cuspidal
representation~$\tau$ of~$\SL_2(k_\F)$ in a Lusztig series with
associated characteristic polynomial~$(X-1)(X+1)^2$. Thus~$\tau$ is a
representation of dimension~$\frac{q-1}2$, of which there are two. We denote
by~$\l_\pi$ the representation of~$\J_{2,1}$ inflated
from~$\theta_{10}\otimes\tau$ and
put~$\pi=\cind_{\J_{2,1}}^\G\l_\pi$, an irreducible cuspidal
representation of~$\G$. 

Following the recipe in Section~\ref{S.synthesis}, we find
that~$s_\rho(\pi)\in\left\{0,\pm\frac 12\right\}$ unless~$\rho$ is a
character of~$\GL_1(\F)$. On the other hand, we get
\[
m_{X-1}^{(1)}=1,\ m_{X-1}^{(2)}=0,\qquad m_{X+1}^{(1)}=0,\ m_{X+1}^{(2)}=1,
\]
and hence
\[
\{s_\pi(\boI),s_\pi(\om_0)\}=\{2,1\},\qquad 
\{s_\pi(\om_1),s_\pi(\om_2)\}=\{1\};
\]
thus~$\IRed(\pi)$ is the
multiset~$\{\!\{([\boI],2),([\boI],1),([\om_1],1),([\om_1],1)\}\!\}$. In this case
we know more since the Langlands parameter~$\varphi_\pi$ corresponding
to~$\pi$ has image in~$\SO_7(\CC)$ so, in particular, has
determinant~$1$; thus it must be
\[
\varphi_\pi=\boI\otimes(\st_3\oplus\st_1)\oplus\om_0\oplus\om_1\oplus\om_2,
\]
since exchanging~$\boI$ and~$\om_0$ would give a representation with
determinant~$\om_0$.

In the notation above, we have~$\E(\pi)=\{\boI,\om_0,\om_1,\om_2\}$
so that~$e(\pi)=4$. Thus the~$\L$-packet containing~$\pi$ consists
of~$16$ irreducible representations,~$8$ of which are cuspidal. On the other hand,
we have~$\Q(\pi)=\{X+1,X-1\}$ so that~$q(\pi)=2$, and~$\d(\pi)=1$, so
that we can construct exactly the cuspidal representations in
the~$\L$-packet, as follows:
\begin{enumerate}
\item There is one rational conjugacy class~$(s)$
in~$\SO_3(k_\F)$ such that its characteristic polynomial
is~$(X-1)(X+1)^2$ \emph{and} the corresponding Lusztig
series~$\Ee(\SL_2(k_\F),s)$ contains two 
irreducible cuspidal representations~$\tau,\tau'$.

Now we can inflate the representations~$\theta_{10}\otimes\tau$
and~$\theta_{10}\otimes\tau'$ of~$\Sp_4(k_\F)\times\SL_2(k_\F)$ to
either~$\J_{2,1}$ or~$\J_{1,2}$ and then induce to~$\G$. This gives us
four inequivalent irreducible cuspidal representations of~$\G$ (one of
which is~$\pi$).
\item There is one rational conjugacy class~$(s_1)$
in~$\SO_7(k_\F)$ such that its characteristic polynomial
is~$(X-1)^5(X+1)^2$ \emph{and} the corresponding Lusztig
series~$\Ee(\Sp_6(k_\F),s_1)$ contains two
irreducible cuspidal representation~$\tau_1,\tau_1'$.

We inflate these representations to either~$\J_{3,0}$ or~$\J_{0,3}$
and induce to~$\G$, giving us another four inequivalent irreducible cuspidal
representations of~$\G$, also inequivalent to those in (i).
\end{enumerate}
\end{example}

\begin{example}\label{ex:no}
Now we look at the simplest example where the information we have so
far is only sufficient to recover the cuspidal representations in
a union of two~$\L$-packets. We take~$\G=\Sp_4(\F)$ and begin with the parahoric
subgroup~$\J_{1,1}$, which has reductive
quotient~$\Gg_{1,1}\simeq\SL_2(k_\F)\times\SL_2(k_\F)$.

We take irreducible cuspidal representations~$\tau_1,\tau_2$ of~$\SL_2(k_\F)$ each
in a Lusztig series with associated characteristic
polynomial~$(X-1)(X+1)^2$, as in the previous example. We denote
by~$\l_\pi$ the representation of~$\J_{1,1}$ inflated
from~$\tau_1\otimes\tau_2$ and put~$\pi=\cind_{\J_{1,1}}^\G\l_\pi$,
an irreducible cuspidal representation of~$\G$. Following the recipe,
this time we obtain
\[
\IRed(\pi)=\{\!\{([\boI],1),([\om_1],2)\}\!\}
\]
and, using the fact that the corresponding Langlands
parameter~$\varphi_\pi$ has determinant~$1$, we have
\begin{equation}\label{eqn:whichomega}
\varphi_\pi=\boI\oplus\om\otimes(\st_3\oplus\st_1),
\end{equation}
where~$\om$ is either~$\om_1$ or~$\om_2$. However, without further
work, we cannot distinguish which ramified quadratic character occurs
here. This reflects the fact that~$m_{X+1}^{(1)}=m_{X+1}^{(2)}=1$ so
that~$\d(\pi)=2$. 

Thus, at this stage, we can only identify the~$4$ cuspidal
representations occurring in the union of the two~$\L$-packets
corresponding to~$\om=\om_1,\om_2$ in~\eqref{eqn:whichomega}: they are
given by independently choosing the~$\tau_i$ to be one of the two
irreducible cuspidal representations of~$\SL_2(k_\F)$ of dimension~$\frac{q-1}2$.

Distinguishing these two~$\L$-packets (and identifying the two
discrete series representation in each of them) requires further
analysis: for this particular example, this is carried out in~\cite{BHS2}.
\end{example}

In general, distinguishing the representations when we have two~$\L$-packets
as in Example~\ref{ex:no} will probably require, as a first step, the
classification of quadratic-unipotent irreducible cuspidal representations of
finite classical groups, and the compatibility of this classification with
Deligne--Lusztig induction, which is done by Waldspurger in~\cite{Wald}.

\subsection{Unramified unitary groups}
Suppose now that~$\G$ is an unramified unitary group of
dimension~$2N+N^\an$. On the one hand
this case is simpler, and we will see that the set of representations
with given inertial reducibility (multi)set is a
single~$\L$-packet. On the other hand, we cannot arbitrarily exchange
the integers~$m_P^{(1)},m_P^{(2)}$ as above, due to parity constraints
-- swapping would sometimes lead to representations of the isometry
group of a non-isometric hermitian space.

Recall that~$\pi$ is induced from the inflation of an irreducible cuspidal
representation~$\tau_\pi^{(1)}\otimes\tau_\pi^{(2)}$ 
of~$\Gg_{N_1}^{(1)}\times\Gg_{N_2}^{(2)}$,
where~$\Gg_{N_i}^{(i)}=\U(\ov\V_{(i)})$, with~$\ov\V_{(i)}$ a
hermitian space of dimension~$2N_i+N_i^\an$. Moreover,
if~$\Gg_{N'_1}^{(1)}\times\Gg_{N'_2}^{(2)}$ is the reductive quotient
of another maximal parahoric subgroup of~$\G$ then, for~$i=1,2$, 
the corresponding space~$\ov\V'_{(i)}$ must have dimension of the
same parity as that of~$\ov\V_{(i)}$.

Recall also that we have a semisimple element~$s_i$ of the dual group of~$\Gg_{N_i}^{(i)}$, such
that~$\tau_\pi^{(i)}$ is the (unique) irreducible cuspidal representation in the
corresponding Lusztig series, and that~$s_i$ has characteristic
polynomial
\[
\prod_P P(X)^{a_P^{(i)}},
\]
where the product runs over all irreducible~$k_\F/k_\so$-self-dual
monic polynomials in~$k_\F[X]$, and~$a_P^{(i)}=\frac 12
m_P^{(i)}(m_P^{(i)}+1)$. Since the degree of each such polynomial~$P$ is
odd, we have
\[
\deg\(P(X)^{a_P^{(i)}}\)\text{ is }\begin{cases}
\text{odd }&\text{ if }m_P^{(i)}\equiv 1,2\pmod 4,\\
\text{even }&\text{ if }m_P^{(i)}\equiv 0,3\pmod 4.
\end{cases}
\]
Thus, if one~$m_P^{(i)}$ is~$1,2\pmod 4$ and the other is~$0,3\pmod
4$, then~$m_P^{(1)}$ cannot be exchanged with~$m_P^{(2)}$
independently of other changes. This is exactly reflected in the
(expected) size of the~$\L$-packet as follows.

We put
\[
\Q_0(\pi)\ =\ \left\{P\in\Q(\pi) \left\vert\ 
\left\lceil{m_P^{(1)}}/2\right\rceil \nequiv
\left\lceil{m_P^{(2)}}/2\right\rceil \pmod{2}
\right.\right\}.
\]
We saw in Remark~\ref{rmk:Q=E} that, in this case, we have~$q(\pi)=e(\pi)$. 
Moreover, the formula for the reducibility points shows
that~$P\in\Q_0(\pi)$ if and only if one
of~$s_\pi(\rho_P),s_\pi(\rho'_P)$ is an odd integer; thus~$\Q_0(\pi)$ is empty if
and only if~$e_0(\pi)=0$. 

Now suppose we are given a map~$\bee:\Q(\pi)\to\{1,2\}$ \emph{such
that~$\#\{P\in\Q_0(\pi)\mid\bee(P)=2\}$ is even}; for
irreducible~$k_\F/k_\so$-self-dual monic~$P\not\in\Q(\pi)$, we
set~$\bee(P)=1$. Then we can find, for~$i=1,2$, a semisimple
element~$s_{\bee}^{(i)}$ in a suitable unitary
group~$\U(2N'_i+N_i^\an,k_\F)$ with characteristic polynomial
\[
\prod_P P(X)^{a_P^{(i)}(\bee)},
\]
where~$a_P^{(i)}(\bee)=\frac 12 m_P^{(i)}(\bee)(m_P^{(i)}(\bee)+1)$ and
\[
m_P^{(i)}(\bee)\ =\ m_P^{(i\cdot\bee(P))},
\]
with the index understood modulo~$3$. Correspondingly, we have
a unique irreducible cuspidal representation~$\tau_\bee=\tau_\bee^{(1)}\otimes\tau_\bee^{(2)}$
of~$\Gg_{N_1',N_2'}$ and, by inflation and compact induction, a unique
irreducible cuspidal representation of~$\G$. 

In this way, we construct~$2^{q(\pi)-e_0(\pi)}$ inequivalent irreducible cuspidal
representations of~$\G$ with the same inertial reducibility set
as~$\pi$, which is exactly the number of cuspidal representations in
the~$\L$-packet of~$\pi$.

\subsection{Special orthogonal and ramified unitary groups}
A similar analysis can be made in the cases of special orthogonal and
unitary groups~$\G$. The constraints for the maps~$\bee$ are like
those in the unramified case, since the anisotropic dimensions of the
groups~$\Gg_{N_i}^{(i)}$ are determined by the group~$\G$, as is the 
sum of the dimensions of the spaces on which~$\Gg_{N_i}^{(i)}$ act. Since the 
details are rather similar to the cases above, we only sketch them.

We are given an irreducible cuspidal representation~$\pi$ of~$\G$. For~$P$ a 
self-dual monic polynomial in~$k_\F[X]$, as in previous cases, we get 
integers~$m_P^{(i)}$, for~$i=1,2$. We modify slightly the definition of~$\Q(\pi)$, 
replacing it with
\[
\begin{cases}
\Q(\pi)\setminus\{X-1\}&\text{ if~$\G$ is even-dimensional ramified unitary;}\\
\Q(\pi)\setminus\{X+1\}&\text{ if~$\G$ is odd-dimensional ramified unitary;}\\
\Q(\pi)\setminus\{X-1,X+1\}&\text{ if~$\G$ is odd-dimensional orthogonal.}
\end{cases}
\]
We also put
\[
\Q'(\pi)=
\begin{cases}
\Q(\pi)&\text{ if~$\G$ is even-dimensional orthogonal;}\\
\Q(\pi)\setminus\{X-1,X+1\}&\text{ otherwise.}
\end{cases}
\]
For~$P$ a self-dual monic polynomial in~$k_\F[X]$, we set
\[
f_P=\begin{cases} 1&\hbox{ if }\deg(P)=1,\\
2&\hbox{ otherwise,}
\end{cases}
\]
and then define
\[
\Q_0(\pi)\ =\ \left\{P\in\Q'(\pi) \left\vert\ 
\left\lceil{m_P^{(1)}}/f_P\right\rceil \nequiv
\left\lceil{m_P^{(2)}}/f_P\right\rceil \pmod{2}
\right.\right\}.
\]
Then we again constrain our map~$\bee:\Q(\pi)\to\{1,2\}$ such
that~$\#\{P\in\Q_0(\pi)\mid\bee(P)=2\}$ is even.
For each such~$\bee$ we can construct a finite set of irreducible cuspidal 
representations of~$\G$. The total number of cuspidal representations obtained 
in this way is one, two or four times the expected number of cuspidal 
representations in the packet, \emph{or half this number}; the latter can occur 
only in the case of even orthogonal groups.

We illustrate this, in particular the last case, with examples, using the same
notation for the quadratic characters of~$\GL_1(\F)$ as in
Examples~\ref{ex:yes} and~\ref{ex:no}.

\begin{example}\label{ex:ortho}
Let~$\G=\SO(\V)$ be the (split) special orthogonal group of an~$8$-dimensional 
orthogonal space~$\V$ with Witt index~$4$. Denote by~$\J_{4,0},\J_{0,4}$ the maximal
compact subgroups whose reductive quotients are~$\SO^+_8(k_\F)$. Denote by~$\tau$ 
the unipotent irreducible cuspidal representation of~$\SO^+_8(k_\F)$, which we may inflate to 
either~$\J_{4,0}$ or~$\J_{0,4}$, and thus obtain irreducible cuspidal 
representations~$\pi=\cind_{\J_{4,0}}^\G\tau$ and~$\pi'=\cind_{\J_{0,4}}^\G\tau$. 
We have
\[
\IRed(\pi)=\IRed(\pi')=\{\!\{([\boI],2),([\boI],2)\}\!\}
\]
and the corresponding Galois parameter has the form
\[
\varphi_\pi=\varphi_{\pi'}=\boI\otimes(\st_3\oplus\st_1)\oplus
\om_0\otimes(\st_3\oplus\st_1).
\]
According to the discussion at the beginning of the section, the packet should 
contain four cuspidal representations, but~$\pi,\pi'$ are the only two cuspidal 
representations with this inertial reducibility set. This disparity comes from 
the difference between the group~$\G$ and the full orthogonal group~$\G^+=\O(\V)$. 
By Proposition~\ref{prop:orthogonalextension}\ref{prop:orth.i}, the 
representation~$\tau$ extends to a representation of~$\O^+_8(k_\F)$, in two ways, 
and inducing the inflation of these two representations from~$\J^+_{4,0}$ 
and~$\J^+_{0,4}$ to~$\G^+$, we obtain four inequivalent irreducible cuspidal 
representations, two restricting to~$\pi$ and the other two to~$\pi'$.

This example illustrates that, for even orthogonal groups, the expected number 
of representations in a packet should be interpreted for the \emph{full} orthogonal 
group, rather than for the special orthogonal group.
\end{example}

\begin{example}\label{ex:orthodd}
Let~$\G=\SO(\V)$ be the special orthogonal group of a~$5$-dimensional
orthogonal space~$\V$ with Witt index~$2$, and denote by$~\J_{2,0}$ the 
maximal compact subgroup whose reductive quotient~$\Gg_{2,0}$
has connected component~$\Gg^\so_{2,0}\simeq\SO^+_{4}(k_\F)\times
\SO_{1}(k_\F)$.

In the dual of the finite group~$\SO^+_{4}(k_\F)$ there is an element~$s$ 
with characteristic polynomial~$(X-1)^2(X+1)^2$ such that the Lusztig 
series~$\Ee(\SO^+_{4}(k_\F),s)$ contains a cuspidal
representation~$\tau$ (in fact, two such representations). The 
inflation of~$\tau$ has two extensions to~$\Gg_{2,0}$ and we denote by~$\l_\pi$ 
the inflation to~$\J_{2,0}$ of one such extension. 
Then~$\pi=\cind_{\J_{2,0}}^\G\l_\pi$ is an irreducible
cuspidal representation of~$\G$, for which
\[
\IRed(\pi)=\{\!\{([\boI],2),([\boI],1),([\om_1],2),([\om_1],1)\}\!\},
\]
and the corresponding Galois parameter~$\vphi_\pi$ has the form
\[
\varphi_\pi=\om\otimes\st_2
\oplus \om'\om_1\otimes\st_2,
\]
for some unramified characters~$\om,\om'$ of order at most~$2$. For
each choice of~$\om,\om'$, the corresponding packet should contain a
unique cuspidal representation, while we have constructed four such 
representations. Thus we have the irreducible cuspidal representations in a union 
of four~$\L$-packets. 
\end{example}

\begin{example}\label{ex:orth}
Let~$\G=\SO(\V)$ be the special orthogonal group of a~$20$-dimensional
orthogonal space~$\V$ with Witt index~$8$ and anisotropic part of
dimension~$4$. Denote by~$\J_{4,4}$
the maximal compact subgroup whose reductive quotient~$\Gg_{4,4}$
has connected component~$\Gg^\so_{4,4}\simeq\SO^-_{10}(k_\F)\times
\SO^-_{10}(k_\F)$.

In the dual of the finite group~$\SO^-_{10}(k_\F)$ there are
elements~$s_1,s_2$ with characteristic polynomials~$(X-1)^8(X+1)^2$
and~$(X-1)^2(X+1)^8$ respectively, and such that the corresponding
Lusztig series~$\Ee(\SO^-_{10}(k_\F),s_i)$ contains a cuspidal
representation~$\tau_i$ (in fact, two such representations). The
representation~$\tau_1\otimes\tau_2$ has two extensions to~$\Gg_{4,4}$
and we denote by~$\l_\pi$ the inflation to~$\J_{4,4}$ of one such
extension. Then~$\pi=\cind_{\J_{4,4}}^\G\l_\pi$ is an irreducible
cuspidal representation of~$\G$, for which
\[
\IRed(\pi)=\{\!\{([\boI],3),([\boI],1),([\om_1],3),([\om_1],1)\}\!\}.
\]
The corresponding Galois parameter~$\vphi_\pi$ has the form
\[
\varphi_\pi=\om\otimes(\st_5\oplus\st_3\oplus\st_1)\oplus\om\om_0
\oplus\om'\om_1\otimes(\st_5\oplus\st_3\oplus\st_1)\oplus\om'\om_2,
\]
for some unramified characters~$\om,\om'$ of order at most~$2$. For
each choice of~$\om,\om'$, the corresponding packet should contain~$8$
irreducible cuspidal representations. 

From the two choices for each of~$\tau_1,\tau_2$ above, and the two
choices of extension to~$\Gg_{4,4}$, we get~$8$
representations. However, we also get~$8$ more by exchanging the roles
of~$\tau_1,\tau_2$, and these also have the same inertial reducibility
(multi)set. However, each of these irreducible cuspidal representations 
has two extensions to the full orthogonal group~$\G^+$. Thus  
we in fact have the irreducible cuspidal representations in the union of 
four~$\L$-packets for the full orthogonal group~$\G^+$.

In this case, there are also~$16$ other representations of the split
special orthogonal group~$\H=\SO(\V')$, where~$\V'$ is
a~$20$-dimensional orthogonal space with Witt index~$10$, with the same inertial
reducibility set, obtained as follows. We denote by~$\J_{8,2}$ a
maximal compact subgroup of~$\H$ whose reductive quotient has connected
component~$\SO^+_{16}(k_\F)\times\SO^+_4(k_\F)$. In the duals of the
isotropic finite groups~$\SO^+_{16}(k_\F)$ and~$\SO^+_4(k_\F)$ there 
are elements~$s_1,s_2$ respectively, with characteristic
polynomials~$(X-1)^8(X+1)^8$
and~$(X-1)^2(X+1)^2$ respectively, such that the corresponding
Lusztig series contain cuspidal
representations~$\tau_1,\tau_2$ respectively (two such
representations in each series). The
representation~$\tau_1\otimes\tau_2$ has two extensions to the
reductive quotient of~$\J_{8,2}$ and, inflating and then inducing
to~$\G$, we obtain an irreducible cuspidal representation. Since we
can also inflate to~$\J_{2,8}$, we obtain~$16$ inequivalent 
representations in this way. Again, each of these representations extends in
two ways to the full orthogonal group~$\H^+$.
\end{example}

\begin{example}\label{ex:uniram}
Let~$\F/\F_\so$ be a ramified quadratic extension and let~$\G=\U(\V)$ be the 
unitary group of a~$14$-dimensional hermitian space~$\V$ with Witt index~$6$ 
and anisotropic part of dimension~$2$. Denote by~$\J_{0,6}$
the maximal compact subgroup whose reductive quotient~$\Gg_{0,6}$
has connected component~$\Gg^\so_{0,6}\simeq\SO^-_{2}(k_\F)\times
\Sp_{12}(k_\F)$.

We fix an irreducible self-dual monic polynomial~$P\in k_\F[X]$ of degree two. 
Then there are semisimple elements~$s_1,s_2$ in the dual groups 
of~$\SO^-_{2}(k_\F)$,~$\Sp_{12}(k_\F)$ respectively, with characteristic 
polynomials~$P(X),P(X)^6$ respectively, such that the corresponding Lusztig 
series contain unique irreducible cuspidal representations~$\tau_1,\tau_2$ 
respectively. Then there is a unique irreducible representation~$\l_\pi$ 
of~$\J_{0,6}$ inflated from a representation of~$\Gg_{0,6}$ 
containing~$\tau_1\otimes\tau_2$ (since~$\tau_1$ does not extend 
to~$\O^-_2(k_\F)$), and~$\pi=\cind_{\J_{0,6}}^\G\l_\pi$ is irreducible and cuspidal.
We have
\[
\IRed(\pi)=\{\!\{([\rho_P],5/2),([\rho_P],1)\}\!\},
\]
where~$\rho_P,\rho'_P$ are the self-dual irreducible cuspidal representations of~$\GL_2(\F)$ 
corresponding to~$P$, and the corresponding Galois parameter has the form
\[
\varphi_\pi=\varphi_P\otimes(\st_4\oplus\st_2)\oplus
\varphi'_P,
\]
where~$\varphi_P,\varphi'_P$ are the Galois parameters corresponding to~$\rho_P,\rho'_P$ 
respectively. The corresponding packet should contain a unique irreducible 
cuspidal representation, which is~$\pi$.

As in Example~\ref{ex:orth}, we also find an irreducible cuspidal representation 
of the~$14$-dimensional quasi-split ramified unitary group with the same inertial 
irreducibility (multi)set, by exchanging the characteristic polynomials of~$s_1,s_2$.
\end{example}

\def\cprime{$'$}



\begin{thebibliography}{10}

\bibitem{Adler}
J.~D. Adler.
\newblock Self-contragredient supercuspidal representations of {${\rm GL}_n$}.
\newblock {\em Proc. Amer. Math. Soc.}, 125(8):2471--2479, 1997.

\bibitem{Arthur}
James Arthur.
\newblock {\em The endoscopic classification of representations}, volume~61 of
  {\em American Mathematical Society Colloquium Publications}.
\newblock American Mathematical Society, Providence, RI, 2013.
\newblock Orthogonal and symplectic groups.

\bibitem{Bl}
Corinne Blondel.
\newblock {${\rm SP}(2N)$}-covers for self-contragredient supercuspidal
  representations of {${\rm GL}(N)$}.
\newblock {\em Ann. Sci. \'Ecole Norm. Sup. (4)}, 37(4):533--558, 2004.

\bibitem{Blb}
Corinne Blondel.
\newblock Repr\'esentation de {W}eil et {$\beta$}-extensions.
\newblock {\em Ann. Inst. Fourier (Grenoble)}, 62(4):1319--1366, 2012.

\bibitem{BHS2}
Corinne Blondel, Guy Henniart, and Shaun Stevens.
\newblock Explicit {$\mathrm L$}-packets for {${\rm Sp}_4(F)$}.
\newblock in preparation.

\bibitem{BHS}
Corinne Blondel, Guy Henniart, and Shaun Stevens.
\newblock Jordan blocks of supercuspidal representations of symplectic groups.
\newblock in preparation.

\bibitem{BHet1}
Colin~J. Bushnell and Guy Henniart.
\newblock The essentially tame local {L}anglands correspondence. {I}.
\newblock {\em J. Amer. Math. Soc.}, 18(3):685--710, 2005.

\bibitem{BHet2}
Colin~J. Bushnell and Guy Henniart.
\newblock The essentially tame local {L}anglands correspondence. {II}.
  {T}otally ramified representations.
\newblock {\em Compos. Math.}, 141(4):979--1011, 2005.

\bibitem{BHet3}
Colin~J. Bushnell and Guy Henniart.
\newblock The essentially tame local {L}anglands correspondence, {III}: the
  general case.
\newblock {\em Proc. Lond. Math. Soc. (3)}, 101(2):497--553, 2010.

\bibitem{BHepi}
Colin~J. Bushnell and Guy Henniart.
\newblock Langlands parameters for epipelagic representations of {${\rm
  GL}_n$}.
\newblock {\em Math. Ann.}, 358(1-2):433--463, 2014.

\bibitem{BHmem}
Colin~J. Bushnell and Guy Henniart.
\newblock To an effective local langlands correspondence.
\newblock {\em Mem. Amer. Math. Soc.}, 231(1087):vi+88, 2014.

\bibitem{BK1}
Colin~J. Bushnell and Philip~C. Kutzko.
\newblock Smooth representations of reductive $p$-adic groups: structure theory
  via types.
\newblock {\em Proc. London Math. Soc. (3)}, 77(3):582--634, 1998.

\bibitem{CE}
Marc Cabanes and Michel Enguehard.
\newblock {\em Representation theory of finite reductive groups}, volume~1 of
  {\em New Mathematical Monographs}.
\newblock Cambridge University Press, Cambridge, 2004.

\bibitem{DeBR}
Stephen DeBacker and Mark Reeder.
\newblock Depth-zero supercuspidal {$L$}-packets and their stability.
\newblock {\em Ann. of Math. (2)}, 169(3):795--901, 2009.

\bibitem{DM}
Fran{\c{c}}ois Digne and Jean Michel.
\newblock {\em Representations of finite groups of {L}ie type}, volume~21 of
  {\em London Mathematical Society Student Texts}.
\newblock Cambridge University Press, Cambridge, 1991.

\bibitem{FS}
Paul Fong and Bhama Srinivasan.
\newblock The blocks of finite classical groups.
\newblock {\em J. Reine Angew. Math.}, 396:122--191, 1989.

\bibitem{GK}
S.~S. Gelbart and A.~W. Knapp.
\newblock {$L$}-indistinguishability and {$R$} groups for the special linear
  group.
\newblock {\em Adv. in Math.}, 43(2):101--121, 1982.

\bibitem{Green}
J.~A. Green.
\newblock The characters of the finite general linear groups.
\newblock {\em Trans. Amer. Math. Soc.}, 80:402--447, 1955.

\bibitem{GReeder}
Benedict~H. Gross and Mark Reeder.
\newblock Arithmetic invariants of discrete {L}anglands parameters.
\newblock {\em Duke Math. J.}, 154(3):431--508, 2010.

\bibitem{HT}
Michael Harris and Richard Taylor.
\newblock {\em The geometry and cohomology of some simple {S}himura varieties},
  volume 151 of {\em Annals of Mathematics Studies}.
\newblock Princeton University Press, Princeton, NJ, 2001.
\newblock With an appendix by Vladimir G. Berkovich.

\bibitem{H00}
Guy Henniart.
\newblock Une preuve simple des conjectures de {L}anglands pour {${\rm GL}(n)$}
  sur un corps {$p$}-adique.
\newblock {\em Invent. Math.}, 139(2):439--455, 2000.

\bibitem{HS}
Kaoru Hiraga and Hiroshi Saito.
\newblock On {$L$}-packets for inner forms of {$SL_n$}.
\newblock {\em Mem. Amer. Math. Soc.}, 215(1013):vi+97, 2012.

\bibitem{Jantzen}
Chris Jantzen.
\newblock Discrete series for {$p$}-adic {${\rm SO}(2n)$} and restrictions of
  representations of {${\rm O}(2n)$}.
\newblock {\em Canad. J. Math.}, 63(2):327--380, 2011.

\bibitem{Kal1}
Tasho Kaletha.
\newblock Endoscopic character identities for depth-zero supercuspidal
  {$L$}-packets.
\newblock {\em Duke Math. J.}, 158(2):161--224, 2011.

\bibitem{Kal4}
Tasho Kaletha.
\newblock {Epipelagic {$L$}-packets and rectifying characters}, September 2012.
\newblock arXiv:1209.1720v1.

\bibitem{Kal3}
Tasho Kaletha.
\newblock Simple wild {$L$}-packets.
\newblock {\em J. Inst. Math. Jussieu}, 12(1):43--75, 2013.

\bibitem{Kal2}
Tasho Kaletha.
\newblock Supercuspidal {$L$}-packets via isocrystals.
\newblock {\em Amer. J. Math.}, 136(1):203--239, 2014.

\bibitem{Kal5}
Tasho Kaletha.
\newblock Regular supercuspidal representations, 2016.

\bibitem{KMSW}
Tasho Kaletha, Alberto M{\'\i}nguez, Sug~Woo Shin, and Paul-James White.
\newblock {Endoscopic Classification of Representations: Inner Forms of Unitary
  Groups}, October 2014.
\newblock arXiv:1409.3731v2.

\bibitem{Kariyama}
Kazutoshi Kariyama.
\newblock On types for unramified {$p$}-adic unitary groups.
\newblock {\em Canad. J. Math.}, 60(5):1067--1107, 2008.

\bibitem{KM}
Philip Kutzko and Lawrence Morris.
\newblock Level zero {H}ecke algebras and parabolic induction: the {S}iegel
  case for split classical groups.
\newblock {\em Int. Math. Res. Not.}, pages Art. ID 97957, 40, 2006.

\bibitem{L77}
G.~Lusztig.
\newblock Irreducible representations of finite classical groups.
\newblock {\em Invent. Math.}, 43(2):125--175, 1977.

\bibitem{Ldisc}
G.~Lusztig.
\newblock On the representations of reductive groups with disconnected centre.
\newblock {\em Ast\'erisque}, (168):10, 157--166, 1988.
\newblock Orbites unipotentes et repr{\'e}sentations, I.

\bibitem{LCBMS}
George Lusztig.
\newblock {\em Representations of finite {C}hevalley groups}, volume~39 of {\em
  CBMS Regional Conference Series in Mathematics}.
\newblock American Mathematical Society, Providence, R.I., 1978.
\newblock Expository lectures from the CBMS Regional Conference held at
  Madison, Wis., August 8--12, 1977.

\bibitem{L}
George Lusztig.
\newblock {\em Characters of reductive groups over a finite field}, volume 107
  of {\em Annals of Mathematics Studies}.
\newblock Princeton University Press, Princeton, NJ, 1984.

\bibitem{MiS}
Michitaka Miyauchi and Shaun Stevens.
\newblock Semisimple types for {$p$}-adic classical groups.
\newblock {\em Math. Ann.}, 358(1-2):257--288, 2014.

\bibitem{Mo1}
C.~M{\oe}glin.
\newblock Sur la classification des s\'eries discr\`etes des groupes classiques
  {$p$}-adiques: param\`etres de {L}anglands et exhaustivit\'e.
\newblock {\em J. Eur. Math. Soc. (JEMS)}, 4(2):143--200, 2002.

\bibitem{Mo2}
Colette M{\oe}glin.
\newblock Points de r\'eductibilit\'e pour les induites de cuspidales.
\newblock {\em J. Algebra}, 268(1):81--117, 2003.

\bibitem{Mo4}
Colette M{\oe}glin.
\newblock Classification des s\'eries discr\`etes pour certains groupes
  classiques {$p$}-adiques.
\newblock In {\em Harmonic analysis, group representations, automorphic forms
  and invariant theory}, volume~12 of {\em Lect. Notes Ser. Inst. Math. Sci.
  Natl. Univ. Singap.}, pages 209--245. World Sci. Publ., Hackensack, NJ, 2007.

\bibitem{Mo3}
Colette M{\oe}glin.
\newblock Classification et changement de base pour les s\'eries discr\`etes
  des groupes unitaires {$p$}-adiques.
\newblock {\em Pacific J. Math.}, 233(1):159--204, 2007.

\bibitem{Mo5}
Colette M{\oe}glin.
\newblock Paquets stables des s\'eries discr\`etes accessibles par endoscopie
  tordue; leur param\`etre de {L}anglands.
\newblock In {\em Automorphic forms and related geometry: assessing the legacy
  of {I}. {I}. {P}iatetski-{S}hapiro}, volume 614 of {\em Contemp. Math.},
  pages 295--336. Amer. Math. Soc., Providence, RI, 2014.

\bibitem{MT}
Colette M{\oe}glin and Marko Tadi{\'c}.
\newblock Construction of discrete series for classical {$p$}-adic groups.
\newblock {\em J. Amer. Math. Soc.}, 15(3):715--786 (electronic), 2002.

\bibitem{Mok}
Chung~Pang Mok.
\newblock Endoscopic classification of representations of quasi-split unitary
  groups.
\newblock {\em Mem. Amer. Math. Soc.}, 235(1108):vi+248, 2015.

\bibitem{M0}
Lawrence Morris.
\newblock Tamely ramified supercuspidal representations of classical groups.
  {I}. {F}iltrations.
\newblock {\em Ann. Sci. \'Ecole Norm. Sup. (4)}, 24(6):705--738, 1991.

\bibitem{MP}
Allen Moy and Gopal Prasad.
\newblock Unrefined minimal {$K$}-types for {$p$}-adic groups.
\newblock {\em Invent. Math.}, 116(1-3):393--408, 1994.

\bibitem{Reeder}
Mark Reeder.
\newblock Supercuspidal {$L$}-packets of positive depth and twisted {C}oxeter
  elements.
\newblock {\em J. Reine Angew. Math.}, 620:1--33, 2008.

\bibitem{RY}
Mark Reeder and Jiu-Kang Yu.
\newblock Epipelagic representations and invariant theory.
\newblock {\em J. Amer. Math. Soc.}, 27(2):437--477, 2014.

\bibitem{Sh90}
Freydoon Shahidi.
\newblock A proof of {L}anglands' conjecture on {P}lancherel measures;
  complementary series for {$p$}-adic groups.
\newblock {\em Ann. of Math. (2)}, 132(2):273--330, 1990.

\bibitem{Silbook}
Allan~J. Silberger.
\newblock {\em Introduction to harmonic analysis on reductive {$p$}-adic
  groups}, volume~23 of {\em Mathematical Notes}.
\newblock Princeton University Press, Princeton, N.J.; University of Tokyo
  Press, Tokyo, 1979.
\newblock Based on lectures by Harish-Chandra at the Institute for Advanced
  Study, 1971--1973.

\bibitem{Sil}
Allan~J. Silberger.
\newblock Special representations of reductive {$p$}-adic groups are not
  integrable.
\newblock {\em Ann. of Math. (2)}, 111(3):571--587, 1980.

\bibitem{S5}
Shaun Stevens.
\newblock The supercuspidal representations of {$p$}-adic classical groups.
\newblock {\em Invent. Math.}, 172(2):289--352, 2008.

\bibitem{Vigneras}
Marie-France Vign{\'e}ras.
\newblock Irreducible modular representations of a reductive {$p$}-adic group
  and simple modules for {H}ecke algebras.
\newblock In {\em European {C}ongress of {M}athematics, {V}ol. {I}
  ({B}arcelona, 2000)}, volume 201 of {\em Progr. Math.}, pages 117--133.
  Birkh\"auser, Basel, 2001.

\bibitem{Wald}
Jean-Loup Waldspurger.
\newblock Une conjecture de {L}usztig pour les groupes classiques.
\newblock {\em M\'em. Soc. Math. Fr. (N.S.)}, (96):vi+166, 2004.

\end{thebibliography}
\end{document}